\documentclass[11pt]{article}

\usepackage{etex}
\usepackage{graphicx}
\usepackage{amsmath,amsthm}
\usepackage{amsfonts}
\usepackage{amssymb}
\usepackage{array}
\usepackage{multirow}
\usepackage{longtable}
\usepackage{setspace}
\usepackage{subfigure}
\usepackage{pgf,tikz}
\usepackage{mathrsfs}
\usetikzlibrary{arrows}
\usetikzlibrary{calc}
\usepackage[rflt]{floatflt}
\usepackage{color, colortbl}
\usepackage{fullpage}
\usepackage[final]{pdfpages}
\usepackage[algo2e,ruled,vlined]{algorithm2e}
\usepackage{hyperref}

\usepackage{verbatim}
\usepackage{enumerate}

\newcommand{\vc}[1]{\ensuremath{\vcenter{\hbox{#1}}}}

\usepackage{color}
\definecolor{blue}{rgb}{0,0,1}

\definecolor{red}{rgb}{1,0,0}

\definecolor{purp}{rgb}{0.55,0.15,0.51}

\usepackage{tikz}
\usepackage{rotating}

\newtheorem{theorem}{Theorem}

\newtheorem{conjecture}[theorem]{Conjecture}
\newtheorem{question}[theorem]{Question}

\newtheorem{claim}[theorem]{Claim}
\newtheorem{problem}[theorem]{Problem}

\theoremstyle{definition}

\newcommand{\greedy}{\rm G}
\newcommand{\alon}{\rm AT}

\newcommand{\mad}{\rm mad}
\newcommand{\chil}{\chi_i^{\ell}}
\usepackage[mathlines]{lineno}
\definecolor{LGrey}{rgb}{.7,.7,.7}
\usepackage{stackengine}

\DeclareMathOperator{\md}{md}

\usetikzlibrary{shapes.geometric}

\tikzset{
externalv/.style={inner sep=2pt, outer sep=0pt, circle,fill=gray},
listsize1/.style={inner sep=1.7pt, outer sep=0pt, circle,fill=black,draw},
listsize2/.style={inner sep=1.7pt, outer sep=0pt, circle,draw},
listsize3/.style={inner sep=1.2pt, outer sep=0pt, regular polygon, regular polygon sides=3, draw},
listsize4/.style={inner sep=1.7pt, outer sep=0pt, regular polygon, regular polygon sides=4, draw},
listsize5/.style={inner sep=1.7pt, outer sep=0pt, regular polygon, regular polygon sides=5, draw},
degree2/.style={inner sep=3.5pt, outer sep=0pt, circle , draw},
externaledges/.style={dashed},
realedges/.style={}
}

\begin{document}

\title{Injective choosability of subcubic planar graphs with girth 6}
\author{
Boris Brimkov \thanks{Department of Computational and Applied Mathematics, Rice University, USA (boris.brimkov@rice.edu).}
\and
Jennifer Edmond \thanks{Mathematics Department, Syracuse University, USA (jledmond@syr.edu).}
\and
Robert Lazar \thanks{Department of Mathematics, Iowa State University, USA (rllazar@iastate.edu).}
\and
Bernard Lidick\'{y} \thanks{Department of Mathematics, Iowa State University, USA (lidicky@iastate.edu).}
\and
Kacy Messerschmidt \thanks{Department of Mathematics, Iowa State University, USA (kacymess@iastate.edu).}
\and
Shanise Walker \thanks{Department of Mathematics, Iowa State University, USA (shanise1@iastate.edu).}
}

\maketitle

\begin{abstract}
An injective coloring of a graph $G$ is an assignment of colors to the vertices of $G$ so that any two vertices with a common neighbor have distinct colors.
A graph $G$ is injectively $k$-choosable if for any list assignment $L$, where $|L(v)| \geq k$ for all $v \in V(G)$, $G$ has an injective $L$-coloring. Injective colorings have applications in the theory of error-correcting codes and are closely related to other notions of colorability.  In this paper, we show that subcubic planar graphs with girth at least 6 are injectively 5-choosable.  
This strengthens the result of  Lu\v{z}ar, \v{S}krekovski, and Tancer that subcubic planar graphs with girth at least 7 are injectively 5-colorable. 
Our result also improves several other results in particular cases.
\end{abstract}

\section{Introduction}

A \emph{proper coloring} of a graph $G$ is an assignment of colors to the vertices of $G$ so that any neighboring vertices receive distinct colors. The \emph{chromatic number} of $G$, $\chi(G)$, is the minimum number of colors needed for a proper coloring of $G$.  An \emph{injective coloring} of a graph $G$ is an assignment of colors to the vertices of $G$ so that any two vertices with a common neighbor receive distinct colors. The \emph{injective chromatic number}, $\chi_i(G)$, is the minimum number of colors needed for an injective coloring of $G$. An injective coloring of $G$ is not necessarily a proper coloring of $G$. Define the \emph{neighboring graph} $G^{(2)}$ by $V(G^{(2)})=V(G)$ and $E(G^{(2)})=\{uv: u \text{ and } v \text{ have a common neighbor in } G\}$. Note that $\chi_i(G)=\chi(G^{(2)})$. 

Injective colorings were first introduced by Hahn, Kratochv\'{i}l, \v{S}ir\'{a}\v{n}, and Sotteau~\cite{HKSS}, where the authors showed injective colorings can be used in coding theory, by relating the injective chromatic number of the hypercube to the theory of error-correcting codes. The authors showed that for a graph $G$ with maximum degree $\Delta$, $\chi_i(G)\leq \Delta(\Delta-1)+1$.  They also showed that computing the injective chromatic number is NP-complete and gave bounds and structural results for the injective chromatic numbers of graphs with special properties.  It is easy to see that $\Delta(G)\leq \chi_i(G)\leq |V(G)|$. 

For each $v \in V(G)$, let $L(v)$ be a set of colors assigned to $v$. Then $L = \{L(v)|v \in V(G)\}$ is a \emph{list assignment} of $G$. Given a list assignment $L$ of $G$, an injective coloring $\varphi$ of  $G$ is called an \emph{injective $L$-coloring} of $G$ if $\varphi(v) \in L(v)$ for every $v\in V(G)$. A graph is injectively \emph{$k$-choosable} if for any list assignment $L$, where $|L(v)| \geq k$ for all $v \in V(G)$, $G$ has an injective $L$-coloring. The \emph{injective choosability number}  of $G$, denoted $\chi_i^{\ell}(G)$, is the minimum  $k$ needed such that $G$ is injectively $k$-choosable. It is clear that $\chi_i(G)\leq \chi_i^{\ell}(G)$. 

  Graphs with low  injective chromatic numbers have been studied extensively. A number of authors have studied the injective chromatic number of graphs $G$ in relation to their maximum degree, $\Delta(G)$, or their maximum average degree,
 $
 \mad(G) = \max_{H \subseteq G}\{ {2|E(H)|}/{|V(H)|} \},
 $
for instance \cite{BI2, BIN, BCRW, LST}. As $\mad(G)<\frac{2g(G)}{g(G)-2}$ for all planar graphs, we can compute a bound for the girth of $G$, $g(G)$, given $\mad(G)$. 
Table~\ref{fig1} consists of results for the injective chromatic number and the injective choosability number of  graphs which depend on planarity, the maximum degree, the maximum average degree, and the girth.
  
\begin{table}[h!] 

\centering
\begin{tabular}{ |p{3cm}|p{1.5cm}|p{1.2cm}|p{1.2cm}|p{1.2cm}|p{6cm}|  }
\hline
Bounds &Planar&	 $\Delta(G)$ & $\mad(G)$ & $g(G)$ & Authors  \\
\hline
$\chi_i(G) \leq \Delta + 1$ & Yes& $\geq 18$ &  & $\geq 6$  &Borodin and Ivanova ~\cite{BI2} \\
\hline 
\rowcolor{LGrey}
$\chi_i(G) \leq \Delta + 3$ & Yes& &  & $\geq 6$ & Dong and Lin ~\cite{DL} \\
\hline
$\chi_i(G) \leq \Delta +3 $ & No & & $<\frac{14}{5}$ & $\geq 7^\ast$ &Doyon, Hahn, and Raspaud ~\cite{DHR}\\
\hline
\rowcolor{LGrey}
$\chi_i(G) \leq \Delta + 4$ & No & & $< 3$ &$\geq 6^\ast$ & Doyon, Hahn, and Raspaud ~\cite{DHR} \\
\hline
$\chi_i(G) \leq \Delta + 8$ & No & & $< \frac{10}{3}$ &$\geq 5^\ast$ & Doyon, Hahn, and Raspaud ~\cite{DHR} \\
\hline
\rowcolor{LGrey}
$\chi_i(G)  \leq 5$ & Yes &  $\leq 3$ &  &  $\geq 7$ & Lu\v{z}ar, \v{S}krekovski, and Tancer~\cite{LST} \\
\hline
$\chil(G) \leq \Delta + 1$ & No& & $<\frac{5}{2}$ & $\geq 10^\ast$ &Cranston, Kim and Yu  ~\cite{CKY2} \\
\hline
$\chil(G) \leq \Delta + 1$ &Yes& $\geq 4$ &  & $\geq 9$ &Cranston, Kim and Yu  ~\cite{CKY2} \\
\hline
$\chil(G) = \Delta $ & Yes& $\geq 4$ & & $\geq 13$ &Cranston, Kim and Yu  ~\cite{CKY2} \\
\hline
$\chil(G) =\Delta$ & No & & $<\frac{42}{19}$ &$\geq 21^\ast$ &Cranston, Kim and Yu~\cite{CKY2}\\
\hline
$\chil(G) \leq 5$ &No & $\geq 3$ & $<\frac{36}{13}$ & $\geq 8^\ast$ &Cranston, Kim and Yu~\cite{CKY1} \\
\hline
$\chil(G) \leq \Delta +2$ & No& $\geq 4$ & $<\frac{14}{5}$ &$\geq 7^\ast$ & Cranston, Kim and Yu ~\cite{CKY1}\\
\hline
$\chil(G) \leq \Delta + 1$ & Yes& $\geq 24$ & & $\geq 6$ &Borodin and Ivanova ~\cite{BI1}\\
\hline
$\chil(G) \leq \Delta + 2$ &Yes & $\geq 12$ &  &$\geq 6$ & Li and Xu ~\cite{LiXu}  \\
\hline 
$\chil(G) \leq \Delta + 2$ & Yes & $\geq 8$ &  & $\geq 6$  &Bu and Lu ~\cite{BuLu} \\
\hline
\rowcolor{LGrey}
$\chil(G) \leq \Delta + 3$ & Yes& &  & $\geq 6$  & Chen and Wu ~\cite{ChenWu} \\
\hline
$\chil(G) \leq \Delta + 4$ & Yes& $\geq 30$ &  & $\geq 5$ &Li and Xu ~\cite{LiXu} \\
\hline
$\chil(G) \leq \Delta + 5$ & Yes& $\geq 18$ &  & $\geq 5$ & Li and Xu ~\cite{LiXu} \\
\hline
$\chil(G) \leq \Delta + 6$ & Yes& $\geq 14$ &  & $\geq 5$ &Li and Xu ~\cite{LiXu} \\
\hline
\hline
$\chil(G) \leq 5$ &Yes & $\leq 3$ & $ $ & $\geq 6$ & This paper \\
\hline

\end{tabular}
\caption{Known results on the injective chromatic number and injective list chromatic number.
 A `Yes' in the `Planar' column indicates that the result holds only for planar graphs, and a `No' indicates that the result holds for both planar and non-planar graphs. 
 A $^\ast$ in the `$g(G)$' column indicates that the girth was obtained using the bound  $\mad(G)<\frac{2g(G)}{g(G)-2}$.
 Results in \cite{CKY2, CKY1} are stated for injective coloring. However, the same proofs work also for injective list coloring.
}
\label{fig1}
\end{table}

For a planar graph $G$ with girth at least $6$ and  any maximum degree $\Delta$, the best known result about the injective chromatic number is $\chi_i(G)\leq \Delta +3$~\cite{DL}. 
In this paper, we improve this result for the case $\Delta=3$.
Moreover,  we improve the result of  Lu\v{z}ar, \v{S}krekovski, and Tancer~\cite{LST} by decreasing the girth condition
and by changing to injective list coloring.
We also improve the other two highlighted bounds in Table \ref{fig1} in special cases.

\begin{theorem}\label{thm:ICN:mainresult}
Every planar graph $G$ with $\Delta (G) \leq 3$ and $g(G) \geq 6$ is injectively 5-choosable.
\end{theorem}

This theorem is a step towards the conjecture of Chen, Hahn, Raspaud, and Wang~\cite{ChenRapsaud} that all planar subcubic  graphs are injectively 5-colorable.
In order to prove Theorem~\ref{thm:ICN:mainresult} we prove a slightly stronger result in Theorem~\ref{thm:ICN:realresult}. 
Let $G$ be a graph and let $L$ be a list assignment.
A \emph{precolored path} in $G$ is a path $P_k$ on $k$ vertices where $|L(v)|=1$ for all $v \in V(P_k)$ and there is at most one vertex $v \in V(P_k)$
with a neighbor in $G-P_k$. Moreover, $\deg_{P_k}(v)$ is maximal among the other
vertices in $P_k$ and $v$ has at most one neighbor in $G-P_k$.
Vertices with lists of size one are called \emph{precolored}.
The set of all precolored vertices $\mathcal{P}$ in $G$ is \emph{proper} if the lists of precolored vertices
give a proper coloring of $G^{(2)}$ when restricted to $\mathcal{P}$. 
That is, the precolored vertices do not  conflict among themselves.

\begin{theorem}\label{thm:ICN:realresult}
Let $G$ be a plane graph with  $\Delta (G) \leq 3$ and $g(G) \geq 6$.
Let $\mathcal{P} \subseteq V(G)$.
Let $L$ be a list assignment of $G$ such that $|L(v)| \geq 5$ for $v\in V(G)\setminus \mathcal{P}$ and $|L(v)| = 1$ for $v \in \mathcal{P}$.
If the precolored vertices are proper, are all in the same face,  form at most two precolored paths,
each of which is on at most three vertices,   then $G$ is injectively $L$-colorable. 
\end{theorem}

\section{Preliminaries}

The following notation shall be used in the sequel.
A \emph{$k$-vertex} is a vertex of degree $k$.  
We denote the degree of a vertex $v$ by $\mathrm{deg}(v)$.
We denote the set of neighbors of $v$ by $N(v)$ and $N(v) \cup \{v\}$ by $N[v]$.
If we want to stress that the degree or neighborhood is in a particular graph $G$, we use subscript $G$, e.g. $\mathrm{deg}_G(v)$.
A \emph{cut edge} or \emph{bridge} is an edge which, when removed, increases the number of components in $G$. Given $S \subset V$, the induced subgraph $G[S]$ is the subgraph of $G$ whose vertex set is $S$ and whose edge set consists of all edges of $G$ which have both ends in $S$.

The length of a face $f$, denoted by $\ell(f)$, is the length of a closed walk around the boundary of $f$.
This is the same as the number of edges incident to $f$ plus the number of cut edges incident to $f$.
A face of length $\ell$ is called an $\ell$-face.
A graph $G$ is \emph{planar} if it is possible to draw $G$ in the plane without edge crossings;
$G$ is \emph{plane} if it is drawn in the plane without edge crossings.
The set of faces of a plane graph $G$ will be denoted by $F(G)$. We say that a 2-vertex $v$ is \emph{nearby} $f$ if $v$ is adjacent to a vertex which is incident to $f$ but $v$ is not incident to $f$ itself. 
The set of all 2-vertices  incident to a face $f$ will be denoted as $I(f)$: 
\[
I(f) =  \{ v \in V(G) : \mathrm{deg} (v) = 2, v \textrm{ is incident to } f \} .
\]
The set of all 2-vertices nearby to a face $f$ will be denoted as $N(f)$: 
\[
N(f) =  \{ v \in V(G) : \mathrm{deg} (v) = 2, v \textrm{ is nearby } f \} .
\]

\subsection{Overview of Method}
In order to prove Theorem \ref{thm:ICN:realresult}, we use the discharging method.
We start with a minimum counterexample and assign initial charges to all vertices and faces of $G$.
By Euler's formula, the sum of charges of all vertices and faces of $G$ is negative. 
Next, we apply rules that move charges between faces and vertices while preserving the sum of the charges.
By the minimality of $G$, certain subgraphs cannot appear in $G$.
We call these subgraphs \emph{reducible configurations}.
By using the fact that $G$ does not contain any {reducible configurations}, we show that the final charge of every face and every vertex of $G$ is nonnegative,
contradicting that the sum of all charges is negative.
For further details and examples of the discharging method, see~\cite{CW}.

\section{Proof of Theorem~\ref{thm:ICN:realresult}}

In this section, we prove Theorem~\ref{thm:ICN:realresult}.
Let $G$ be a minimum counterexample. 
The minimality of $G$ is defined by
first minimizing the number of connected components of $G$, 
then, subject to that, 
minimizing the number of non-precolored vertices,
and finally, subject to the first two conditions, maximizing the number of precolored vertices.
Recall that $G$ is a plane graph with $\Delta(G) \leq 3$ and $g(G) \geq 6$.
Let $L$ be a list assignment for $G$ such that 
the precolored vertices are proper and form at most two precolored paths, 
each on at most three vertices in the same face of $G$,
and such that there is no  injective $L$-coloring of $G$.
Denote the set of all precolored vertices by $\mathcal{P}$.
Note that there are at most six vertices in $\mathcal{P}$.

By the minimality of $G$, we obtain that $G$ is connected.
In addition, each subgraph of $G$ with fewer non-precolored vertices is injectively $L$-colorable.
Moreover, if $G'$ is a connected graph obtained from $G$ by adding precolored vertices which still satisfies the assumptions of Theorem~\ref{thm:ICN:realresult}, then $G'$ is injectively $L$-colorable.
  
 \subsection{Preliminary observations}
 
We first make some preliminary observations about the structure of $G$.

\begin{claim}\label{claimprecolored}
Every precolored path is on three vertices;
every vertex of $\mathcal{P}$  has degree one or three in $G$;
there are at least two vertices that are not precolored.
\end{claim}
\begin{proof}
If $P$ is a precolored path on less than $3$ vertices, then we can
add a new precolored vertex to one end of $P$, contradicting
the maximality of  $|\mathcal{P}|$.

Suppose for contradiction there exists $v \in \mathcal{P}$ with $\deg_G(v) = 2$.
Since each precolored path $P$ has three vertices and by the assumptions of Theorem~\ref{thm:ICN:realresult}, only the middle vertex of $P$ can be adjacent to a vertex in $G-P$, $v$ is a middle vertex of a precolored path. 
However, $\deg_P(v) = 2$. Since $G$ is connected, we get $P = G$.
Since $\mathcal{P}$ is proper, there exists an injective $L$-coloring of $G$, which is a contradiction.
 
Suppose $v$ is the only non-precolored vertex.
Then $\deg(v) \leq 2$ and $G$ is a tree.
Therefore, $v$ has at most four neighbors in $G^{(2)}$ and $G$ is injectively $L$-colorable, a contradiction.
\end{proof}

\begin{claim}\label{claim1} 
If $v_1$ and $v_2$ are two distinct 2-vertices, then the distance between them is at least four.
\end{claim}
\begin{proof}
Let $v_1$ and $v_2$ be distinct $2$-vertices that are endpoints of a path $Z$ of length $\ell$, where $\ell \leq 3$.
Let $G'$ be obtained from $G$ by removing the vertices of $Z$.
By the minimality of $G$, there exists an injective $L$-coloring $\varphi$ of $G'$.
Observe that  the subgraph of $G^{(2)}$ induced by $V(G')$ is the same as $G'^{(2)}$.

Let $L_Z$ be a list assignment for $Z$ defined in the following way:
\[
L_Z(v) = L(v) \setminus \{\varphi(u): u \in N_{G^{(2)}}(v) \cap V(G') \}.
\]
See Figure~\ref{fig-2vertices} for possible cases based on $\ell$, and refer to Table~\ref{table:configkey} for shape meanings
in Figure~\ref{fig-2vertices} and subsequent figures.
In all three cases, $|L_Z(v_i)| \geq 2$ for $i \in \{1,2\}$ and all the remaining vertices of $Z$
have at least one color available. 
Hence there exists an injective $L_Z$-coloring $\rho$ of $Z$.

\begin{table}
\centering
  \begin{tabular}{ |c|m{1cm} | c|c| }
  \hline 
  Shape & List Size & 2-vertex& With External Edge\\\hline 
   \begin{tikzpicture}[scale=3]  
\node at (0.0,0.0) [listsize1](x0){}; \end{tikzpicture}  & 1& \begin{tikzpicture}[scale=3]  
\node at (0.0,0.0) [listsize1](x0){}; \draw (x0) node[degree2]{};\end{tikzpicture} &  \begin{tikzpicture}[scale=3]  
\node at (0.0,0.0) [listsize1](x0){}; \coordinate (x16) at (.25,0);  \draw[externaledges]  (x0)--(x16); \end{tikzpicture} \\ \hline
   \begin{tikzpicture}[scale=3]  
\node at (0.0,0.0) [listsize2](x0){}; \end{tikzpicture}  & 2&\begin{tikzpicture}[scale=3]  
\node at (0.0,0.0) [listsize2](x0){}; \draw (x0) node[degree2]{};\end{tikzpicture} &  \begin{tikzpicture}[scale=3]  
\node at (0.0,0.0) [listsize2](x0){}; \coordinate (x16) at (.25,0);  \draw[externaledges]  (x0)--(x16); \end{tikzpicture} \\ \hline
     \begin{tikzpicture}[scale=3]  
\node at (0.0,0.0) [listsize3](x0){}; \end{tikzpicture}  & 3& \begin{tikzpicture}[scale=3]  
\node at (0.0,0.0) [listsize3](x0){}; \draw (x0) node[degree2]{};\end{tikzpicture} &  \begin{tikzpicture}[scale=3]  
\node at (0.0,0.0) [listsize3](x0){}; \coordinate (x16) at (.25,0);  \draw[externaledges]  (x0)--(x16); \end{tikzpicture} \\ 
   \hline   \begin{tikzpicture}[scale=3]  
\node at (0.0,0.0) [listsize4](x0){}; \end{tikzpicture}  &4& \begin{tikzpicture}[scale=3]  
\node at (0.0,0.0) [listsize4](x0){}; \draw (x0) node[degree2]{};\end{tikzpicture} &  \begin{tikzpicture}[scale=3]  
\node at (0.0,0.0) [listsize4](x0){}; \coordinate (x16) at (.25,0);  \draw[externaledges]  (x0)--(x16); \end{tikzpicture}  \\ 
   \hline    \begin{tikzpicture}[scale=3]  
\node at (0.0,0.0) [listsize5](x0){}; \end{tikzpicture}  & 5& \begin{tikzpicture}[scale=3]  
\node at (0.0,0.0) [listsize5](x0){}; \draw (x0) node[degree2]{};\end{tikzpicture} &  \begin{tikzpicture}[scale=3]  
\node at (0.0,0.0) [listsize5](x0){}; \coordinate (x16) at (.25,0);  \draw[externaledges]  (x0)--(x16); \end{tikzpicture} \\ 
   \hline
    \end{tabular}
\caption{Key of list sizes.}
\label{table:configkey}
\end{table}

Observe that  $\rho$ and $\varphi$ form an injective $L$-coloring of $G$, which
is a contradiction.
\begin{figure}
\begin{center}
\begin{tikzpicture}[scale=1]  
\node at (0,0)[listsize2,label=above:$v_1$](x0){};
\node at (1,0) [listsize2,label=above:$v_2$](x1){};
\node at (0.5,-0.4) {$Z$};
\draw (x0) node[degree2]{} (x1) node[degree2]{};
\draw[realedges]  (x0)--(x1);
\draw[externaledges]
  (x0)-- ++(180:1) node[externalv](a){}
  (a) -- ++(90:1) node[externalv]{}
  (a) -- ++(270:1) node[externalv]{}
  (x1)-- ++(0:1) node[externalv](a){}
  (a) -- ++(90:1) node[externalv]{}
  (a) -- ++(270:1) node[externalv]{}
;  
\end{tikzpicture}
\hskip 2em
\begin{tikzpicture}[scale=1]  
\node at (0,0)[listsize2,label=above:$v_1$](x0){};
\node at (1,0)[listsize1](y0){};
\node at (2,0) [listsize2,label=above:$v_2$](x1){};
\node at (1,-0.4) {$Z$};
\draw (x0) node[degree2]{} (x1) node[degree2]{};
\draw[realedges]  (x0)--(x1);
\draw[externaledges]
  (x0)-- ++(180:1) node[externalv](a){}
  (a) -- ++(90:1) node[externalv]{}
  (a) -- ++(270:1) node[externalv]{}
  (x1)-- ++(0:1) node[externalv](a){}
  (a) -- ++(90:1) node[externalv]{}
  (a) -- ++(270:1) node[externalv]{}
  (y0)-- ++(90:1) node[externalv](a){}
  (a) -- ++(45:1) node[externalv]{}
  (a) -- ++(135:1) node[externalv]{}
;  
\end{tikzpicture}
\hskip 2em
\begin{tikzpicture}[scale=1]  
\node at (0,0)[listsize2,label=above:$v_1$](x0){};
\node at (1,0)[listsize1](y0){};
\node at (2,0)[listsize1](y1){};
\node at (3,0) [listsize2,label=above:$v_2$](x1){};
\node at (1.5,-0.4) {$Z$};
\draw (x0) node[degree2]{} (x1) node[degree2]{};
\draw[realedges]  (x0)--(x1);
\draw[externaledges]
  (x0)-- ++(180:1) node[externalv](a){}
  (a) -- ++(90:1) node[externalv]{}
  (a) -- ++(270:1) node[externalv]{}
  (x1)-- ++(0:1) node[externalv](a){}
  (a) -- ++(90:1) node[externalv]{}
  (a) -- ++(270:1) node[externalv]{}
  (y0)-- ++(90:1) node[externalv](a){}
  (a) -- ++(90:1) node[externalv]{}
  (a) -- ++(135:1) node[externalv]{}
  (y1)-- ++(90:1) node[externalv](a){}
  (a) -- ++(45:1) node[externalv]{}
  (a) -- ++(90:1) node[externalv]{}
;  
\end{tikzpicture}
\end{center}
\caption{Path $Z$ connecting two vertices of distance at most three in Claim~\ref{claim1}. Dashed edges and gray vertices correspond to (possible) edges and vertices of $G$ outside of  $Z$.}\label{fig-2vertices}
\end{figure}
\end{proof}

 \begin{claim}\label{cl:bridge}
 If $e=uv$ is a bridge in $G$, then $u$ or $v$ is in $\mathcal{P}$.
 \end{claim}
 \begin{proof}
 Suppose for contradiction that $e=uv$ is a bridge and neither $u$ nor $v$ is precolored.
 Denote the two connected components of $G-e$ by $X_u$ and $X_v$ where $u \in V(X_u)$ and $v \in V(X_v)$.
 Moreover, if possible, pick $e$ such that $X_v$ does not contain any precolored vertices.
 
 First we show that each of $u$ and $v$ have a neighbor in $\mathcal{P}$.
  To show this, assume that either $X_v$ does not contain any precolored vertices,
 or if both $X_u$ and  $X_v$ contain precolored vertices, that $v$ is not adjacent
 to any of them.
 
 By the minimality of $G$, there exists an injective $L$-coloring $\varphi$ of $X_u$.
 Let $X'_v = G[X_v \cup N[u]]$.
 We create a list assignment $L'$ for $X'_v$, where $L'(y) = L(y)$ if $y \in V(X_v)$ and $L'(y) = \{\varphi(y)\}$ if $y \in V(X_u)$.

 Observe that $X'_v$ with $L'$ is a plane graph with at most two precolored paths,
 one on the vertices $N[u] \cap X_u$ and possibly another one in $X_v$. 
 Moreover, the set of precolored vertices in $L'$ is
 proper since $v$ is not adjacent to precolored vertices in $X_v$.
 Finally, if there are two precolored paths in $X'_v$, they must both be  part of the outer face $F_o$ since the two precolored paths in $G$ are in $F_o$. 
 Hence $e$ is also in $F_o$.
 
 By the minimality of $G$, $X'_v$ has an injective $L'$-coloring $\rho$.
 Since $\rho$ and $\varphi$ agree on $u$ and its neighbors in $X_u$, it is possible
 to combine $\varphi$ and $\rho$ into an injective $L$-coloring of $G$, which is a contradiction.
 
 Hence we conclude that each of $u$ and $v$ have a neighbor in $\mathcal{P}$.
 By Claim~\ref{claim1}, $u$ cannot be a 2-vertex. 
 Then $u$ is a 3-vertex and has a non-precolored neighbor $w$ distinct from $v$.
 Since $uv$ is a bridge, $uw$ is also a bridge. 
 Since only two vertices in $G-\mathcal{P}$ have neighbors in $\mathcal{P}$, we  get
 a contradiction with our choice of $e$ since $vw$ is a bridge and 
 the connected component of $G-vw$ containing $w$ has no precolored vertices.
 \end{proof}

\begin{claim}\label{cl:noP2}
If a vertex $v$ has two precolored neighbors, then $v$ is also precolored.
\end{claim}
\begin{proof}
Suppose for contradiction that $v$ is not precolored.
If $v$ is a vertex of degree two, the connectivity of $G$ implies that     $v$ is the only non-precolored vertex, contradicting Claim~\ref{claimprecolored}.
If $v$ is a vertex of degree three, then consider its non-precolored neighbor $w$.
The edge $vw$ has no endpoint precolored and it is a bridge, contradicting Claim~\ref{cl:bridge}.
\end{proof}

\begin{claim}\label{cl:leaf}
If $v$ is a vertex of degree one in $G$, then $v$ is precolored.
\end{claim}
\begin{proof}
Suppose for contradiction that $v$ is a vertex of degree one that is not precolored.
By the minimality of $G$, there exists an injective $L$-coloring $\varphi$ of $G-v$.
Since $v$ has at most 4 neighbors in $G^{(2)}$ and $|L(v)| \geq 5$, it is possible
to extend $\varphi$ to an injective $L$-coloring of $G$, which is a contradiction to the minimality of $G$.
\end{proof}

\begin{claim}\label{cl:face}
Every face of $G-\mathcal{P}$ is bounded by a cycle. 
\end{claim}
\begin{proof}
Claims~\ref{cl:bridge}, \ref{cl:noP2}, and \ref{cl:leaf} imply that $G-\mathcal{P}$ is bridgeless and every face $f$ of $G-\mathcal{P}$ is 
bounded by a cycle of length $\ell(f)$.
\end{proof}

\begin{claim}\label{cl2edgeconnected}
Let $e_1$ and $e_2$ be edges with distinct endpoints 
such that $G-\{e_1,e_2\}$ is disconnected but neither $e_1$ nor $e_2$ is a bridge.
Then each connected component of $G-\{e_1,e_2\}$ contains precolored vertices.
In addition, both $e_1$ and $e_2$ are in the outer face.
\end{claim}
\begin{proof}
Let $e_1$ and $e_2$ be edges  with distinct endpoints such that $G-\{e_1,e_2\}$ is disconnected but neither $e_1$
nor $e_2$ is a bridge.

Since neither $e_1$ nor $e_2$ is a bridge, $G-\{e_1,e_2\}$ contains exactly two connected components $X$ and $Y$.
Suppose for contradiction that $Y$ does not contain any precolored vertices.

By the minimality of $G$, there is an injective $L$-coloring $\varphi$ of $X$.
Let $u_i$ be a vertex of $e_i$ in $V(X)$ for $i \in \{1,2\}$.
Notice that neither $u_1$ nor $u_2$ is precolored since precolored vertices
are incident only to bridges.

Next we build a graph $Y'$ and a list assignment $L'$.
We start with $Y'=Y$  and  define $L'(v) = L(v)$ for all $v \in V(Y')$.
Then for $i\in\{1,2\}$ we add $u_i$ and $e_i$ to $Y'$ and define $L'(u_i) = \varphi(u_i)$.
Then for every $x \in N_X(u_i)$ we add to $Y'$ a new vertex $u_i^x$ adjacent to $u_i$
and define $L'(u_i^x) = \varphi(x)$; see Figure~\ref{fig:claim2}.

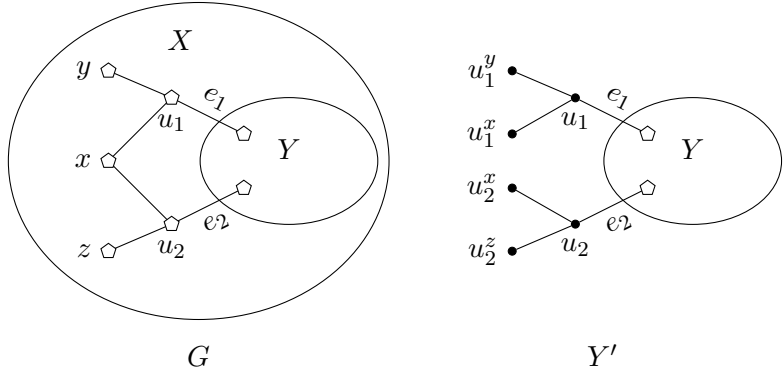
\begin{figure}
\begin{center}
\tikzstyle{Bvertex}=[circle, black, fill, draw, inner sep=0pt, minimum size=3pt]
\newcommand{\Bvertex}{\node[Bvertex]}
	\begin{tikzpicture}[scale=1.2]
		\node[listsize5] (1) at (-.5,.3) []{};
		\node[listsize5] (2) at (-.5, -.3) []{};
		\node[listsize5] (3) at (-1.3, .7) [label=below:$u_1$]{};
		\node[listsize5] (4) at (-1.3, -.7) [label=below:$u_2$]{};
		\node[listsize5] (5) at (-2,0) [label=left:$x$]{};
		\node[listsize5] (6) at (-2,1) [label=left:$y$]{};
		\node[listsize5] (7) at (-2,-1) [label=left:$z$]{};
		\node at (0,-.2)[label=$Y$]{}; 
		\node at (-1.2,1)[label=$X$]{}; 
		\node at (-1,-2.5)[label=$G$]{}; 
		\draw (1) -- (3) node[midway,sloped,above] {$e_1$};
		\draw (5) to (3);
		\draw (6) to (3);
		\draw (2) -- (4) node[midway,sloped,below] {$e_2$};
		\draw (5) to (4);
		\draw (7) to (4);
		\draw (0,0) ellipse (28pt and 20pt) {};
		\draw (-1,0) ellipse (60pt and 50pt);
	\end{tikzpicture}
	\hskip 2em
		\begin{tikzpicture}[scale=1.2]
	  \node[listsize5] (1) at (-.5,.3) []{};
		\node[listsize5] (2) at (-.5, -.3) []{};
		\Bvertex (3) at (-1.3, .7) [label=below:$u_1$]{};
		\Bvertex (4) at (-1.3, -.7) [label=below:$u_2$]{};
		\Bvertex (5) at (-2,0.3) [label=left:$u_1^x$]{};
		\Bvertex (8) at (-2,-0.3) [label=left:$u_2^x$]{};
		\Bvertex (6) at (-2,1) [label=left:$u_1^y$]{};
		\Bvertex (7) at (-2,-1) [label=left:$u_2^z$]{};
		\node at (0,-.2)[label=$Y$]{}; 
		\node at (-1,-2.5)[label=$Y'$]{}; 
		\draw (1) -- (3) node[midway,sloped,above] {$e_1$};
		\draw (5) to (3);
		\draw (6) to (3);
		\draw (2) -- (4) node[midway,sloped,below] {$e_2$};
		\draw (8) to (4);
		\draw (7) to (4);
		\draw (0,0) ellipse (28pt and 20pt) {};
	\end{tikzpicture}
\end{center}
\caption{Dealing with edges $e_1$ and $e_2$ that form a cut in Claim~\ref{cl2edgeconnected}.}\label{fig:claim2}
\end{figure}

Notice that if $u_1$ and $u_2$ have a common neighbor $x$, there
are two vertices $u_1^x$ and $u_2^x$ corresponding to $x$ in $Y'$.
Similarly, if  $u_1$ and $u_2$ are adjacent in $G$, 
we have edges $u_1u_1^{u_2}$ and $u_2u_2^{u_1}$ instead of the edge $u_1u_2$ in $Y'$.
We do this to keep the assumption
that precolored vertices form two paths. Since $e_1$ and $e_2$ do not share
vertices, the precolored vertices in $Y'$ with $L'$ are proper.

By the minimality of $G$, there is an injective $L'$-coloring $\rho$ of $Y'$.
Since  $\varphi$ and $\rho$ agree on $u_1$, $u_2$ and their neighbors in $X$,
it is possible to combine $\varphi$ and $\rho$ and obtain an injective $L$-coloring
of $G$, which is a contradiction.
\end{proof}

\subsection{Reducible Configurations}\label{secreducible}
A \emph{configuration} is a pair $(H,\md)$, where $H$ is a plane graph and $\md$ is a mapping $\md: V(H) \rightarrow \mathbb{N}$.
The notation $\md$ stands for maximum degree.

A \emph{basic reducible configuration} is a configuration $(H,\md)$, where $H$ is one of the plane graphs $C_{2,1},\ldots,C_{10,3}$ depicted in Figures~\ref{fig-2x}--\ref{fig-10x} and $\md$ of a vertex $v$ can be read from the
figure of $H$ by adding the degree of $v$ and the number of incident dashed edges.
We denote the set of basic reducible configurations by $\mathcal{B}$.

Also note the graphs in Figures~\ref{fig-2x}--\ref{fig-10x} are embedded in the plane and each face, except the outer face, bounded by a cycle $C$ is labeled with a number $\ell$ equal to the size of $C$.
If we are dealing with the outer face $F_o$, the cycle $C$ does not form the entire boundary of $F_o$ as 
$F_o$ could also contain precolored vertices. 
In this case, $\ell$ corresponds to the length of the outer face of $G-\mathcal{P}$. 
Figures~\ref{fig-2x}--\ref{fig-10x} are moved to the next section in order to make it easier
for the reader to find them when they are actually needed.

Next we obtain the set of \emph{reducible configurations} $\mathcal{R}$
by taking $\mathcal{B}$ together with
configurations obtained from $(H,\md) \in \mathcal{B}$; these are obtained by identification 
of two vertices of degree 1 into a new vertex $w$ and defining $\md(w)=2$.
We  keep in $\mathcal{R}$ only configurations with plane graphs of girth at least 6.

We say that $(H,\md)$ \emph{appears} in $G$ if $G$ contains a subgraph $H'$ isomorphic to $H$ and
for every vertex $v$ of $H'$, $\deg_G(v) \leq \md(v)$, where $\md$ is defined on $H'$ by its isomorphism to $H$.

We plan to show that no reducible configuration $(H,\md)$ appears in $G$.
If $(H,\md)$ appears in $G$, we wish to obtain a contradiction by finding an injective $L$-coloring $\varphi$ of  $G-H'$, where $H'$ is the isomorphic copy of $H$, and then by extending $\varphi$ to an injective $L$-coloring to $H'$.
To this end, we need to consider the subgraph $W$ of $G^{(2)}$ induced by vertices of $H'$.
It may happen that $W$ is not isomorphic to $H'^{(2)}$ if
$H'$ is not an induced subgraph of $G$ or some of the vertices in $H'$ have a common neighbor that is not in $H'$. 
We cover these cases by expanding the set of reducible configurations, as follows.

Let $\mathcal{A}$ be obtained from $\mathcal{R}$ by possibly repeating any of the following operations
to configurations $(H,\md)$ already in $\mathcal{A}$:
\begin{itemize}
\item Add an edge between two vertices $u$ and $v$ where $\md(u) > \deg(u)$ and $\md(v) > \deg(v)$.
\item Add a new vertex $w$ adjacent to $u$ and $v$ where $\md(u) > \deg(u)$ and $\md(v) > \deg(v)$ and let $\md(w)=3$.
\end{itemize}
Notice that each operation decreases the number of vertices $v$ in $H$ where $\deg(v) < \md(v)$, hence $\mathcal{A}$ is finite.
We call $\mathcal{A}$ \emph{all reducible configurations}.
Note also that if $(H,\md)\in \mathcal{R}$ appears in $G$ then there exists $(H',\md')\in \mathcal{A}$
such that $G$ contains an isomorphic copy $X$ of $H'$
and  $H'^{(2)}$ is isomorphic to $G^{(2)}$ restricted to the vertices of $X$. 

Let $\mathcal{E}$ be the configurations from $\mathcal{A}$ containing $X_1$ or $X_2$ that are depicted in Figure~\ref{fig-E}.
The configurations in $\mathcal{E}$ are not injectively colorable from the depicted lists.
We call these configurations \emph{exceptions}.

\begin{figure}[h!]
\[
\begin{array}{cccc}
\vc{
\begin{tikzpicture}[scale=2]  
\node at (0.0,0.0) [listsize4](x0){};
\node at (-0.4,0) [listsize4](x1){};
\node at (-0.6,-0.35) [listsize4](x2){};
\node at (-0.6,-0.7) [listsize4](x3){};
\node at (-0.4,-1) [listsize2](x4){};
\node at (-0.0,-1) [listsize1](x5){};
\node at (0.2,-0.7) [listsize1](x6){};
\node at (0.2,-0.35) [listsize2](x7){};
\node at (0.2,0.4) [listsize4](x8){};
\node at (0,0.8) [listsize3](x9){};
\node at (-0.4,0.8) [listsize2](x10){};
\node at (-0.6,0.4) [listsize2](x11){};
\node at (-1,0.8) [listsize3](x12){};
\coordinate (x13) at (-0.5,-1.2);
\coordinate (x14) at (0.1,-1.2);
\coordinate (x15) at (0.4,-0.8);
\coordinate (x16) at (0.4,-0.35);
\coordinate (x17) at (-1.3,0.8);
\coordinate (x18) at (-0.5,1);
\coordinate (x19) at (-0.8,0.4);
\draw (x3) node[degree2]{} (x8) node[degree2]{};
\draw[realedges]  (x0)--(x1) (x0)--(x7) (x0)--(x8) (x1)--(x2) (x1)--(x11) (x2)--(x3) (x3)--(x4) (x4)--(x5) (x5)--(x6) (x6)--(x7) (x8)--(x9) (x9)--(x10) (x10)--(x11);
\draw[externaledges]   (x4)--(x13) (x5)--(x14) (x6)--(x15) (x7)--(x16) (x12)--(x17) (x10)--(x18) (x11)--(x19);\
\draw[line width=2pt] (x1)--node[pos=0.5,label=left:$e_1$]{}(x11) (x10)--node[pos=0.5,label=below:$e_2$]{}(x9);
\draw (x2) to[out=180,in=270] (x12) (x12)to[out=90,in=90]  (x9);
\node at (-0.2,-0.5){ 8 };
\node at (-0.2,.35){6};   
\end{tikzpicture}
}
&
\vc{
\begin{tikzpicture}[scale=2]  
\node at (0.0,0.0) [listsize4](x0){};
\node at (-0.4,0) [listsize4](x1){};
\node at (-0.6,-0.35) [listsize4](x2){};
\node at (-0.6,-0.7) [listsize4](x3){};
\node at (-0.4,-1) [listsize2](x4){};
\node at (-0.0,-1) [listsize1](x5){};
\node at (0.2,-0.7) [listsize1](x6){};
\node at (0.2,-0.35) [listsize2](x7){};
\node at (0.2,0.4) [listsize4](x8){};
\node at (0,0.8) [listsize3](x9){};
\node at (-0.4,0.8) [listsize2](x10){};
\node at (-0.6,0.4) [listsize2](x11){};
\node at (-1,0.8) [listsize3](x12){};
\coordinate (x13) at (-0.5,-1.2);
\coordinate (x14) at (0.1,-1.2);
\coordinate (x15) at (0.4,-0.8);
\coordinate (x16) at (0.4,-0.35);
\coordinate (x17) at (-0.7,0.8);
\coordinate (x18) at (-0.5,1);
\coordinate (x19) at (-0.8,0.4);
\draw (x3) node[degree2]{} (x8) node[degree2]{};
\draw[realedges]  (x0)--(x1) (x0)--(x7) (x0)--(x8) (x1)--(x2) (x1)--(x11) (x2)--(x3) (x3)--(x4) (x4)--(x5) (x5)--(x6) (x6)--(x7) (x8)--(x9) (x9)--(x10) (x10)--(x11);
\draw[externaledges]   (x4)--(x13) (x5)--(x14) (x6)--(x15) (x7)--(x16) (x12)--(x17) (x10)--(x18) (x11)--(x19);
\draw[line width=2pt] (x2)--node[pos=0.5,label=left:$e_1$]{}(x3) (x0)--node[pos=0.5,label=right:$e_2$]{}(x7);
\draw (x2) to[out=180,in=270] (x12) (x12)to[out=90,in=90]  (x9);
\node at (-0.2,-0.5){ 8 };
\node at (-0.2,.35){6}; 
\end{tikzpicture}
}
&
\vc{
\begin{tikzpicture}[scale=2]  
\clip (-0.8,-1.2) rectangle (0.8,0.7);
\node at (0.0,0.0) [listsize3](x0){};
\node at (-0.35,0) [listsize3](x1){};
\node at (-.5,-0.5) [listsize2](x2){};
\node at (-0.35,-1) [listsize3](x3){};
\node at (0.0,-1) [listsize3](x4){};
\node at (0.0,-0.5) [listsize5](x5){};
\node at (0.35,0) [listsize3](x6){};
\node at (0.5,-0.5) [listsize2](x7){};
\node at (0.35,-1) [listsize3](x8){};
\node at (0,0.5) [listsize3](x15){};
\node at (-0.6,0.3) [listsize1](x9){};
\node at (0.6,0.3) [listsize1](x12){};
\coordinate (x10) at (0,0.7);
\coordinate (x11) at (-0.6,-1);
\coordinate (x13) at (.75,-.5);
\coordinate (x14) at (.6,-1);
\draw (x5) node[degree2]{};
\draw[realedges]  (x0)--(x1) (x0)--(x5) (x0)--(x6) (x1)--(x2) (x2)--(x3) (x3)--(x4) (x4)--(x5) (x4)--(x8) (x6)--(x7) (x7)--(x8);
\draw[externaledges]  (x1)--(x9) (x15)--(x10) (x3)--(x11) (x6)--(x12) (x8)--(x14);
\draw[realedges] (x7) to[out=60,in=0,looseness=2] (x15) (x2) to[out=120,in=180,looseness=2] (x15); 
\draw[line width=2pt] 
 (x6)--node[pos=0.5,label=left:$e_2$]{}(x12)
 (x1)--node[pos=0.5,label=right:$e_1$]{}(x9);
\node at (-.25,-.5){ 6};  
\node at (.25,-.5){ 6};    
\end{tikzpicture}
}
\\[10pt]
X_1 & X_1 & X_2  
\end{array}
\]
\caption{Exceptional graphs $X_1$ and $X_2$.}\label{fig-E}
\end{figure}
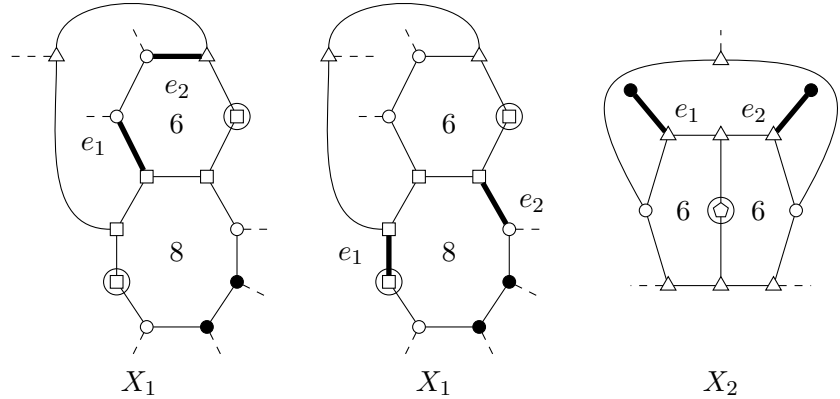

\begin{claim}\label{cl:reducible}
If $(H,\md)\in \mathcal{A}\setminus\mathcal{E}$ then $(H,\md)$ does not appear in $G$.
\end{claim}
\begin{proof}
Suppose for contradiction that some $(H,\md)\in \mathcal{A}$ appears in $G$.
Let $H'$ be the isomorphic copy of $H$ in $G$.
Assume that $H$ is as large as possible. That means $H'$ is an induced subgraph of $G$
and no pair of vertices of $H'$ have a common neighbor in $G-H'$.

By the minimality of $G$, there exists an injective $L$-coloring $\varphi$ of $G-H'$.
Create a list assignment $L'$ from $L$ by removing the colors used on neighbors in $G^{(2)}-H'$, that is  for every $v \in V(H')$,
\[
L'(v) = L(v) \setminus \{ \varphi(x):  vx \in E(G^{(2)})\text{ and } x \in  V(G)\setminus V(H') \}.
\]
We depict $|L'(v)|$ for configurations in $\mathcal{B}$ in Figures~\ref{fig-2x}--\ref{fig-10x} by the shape of vertices.
Refer to Table~\ref{table:configkey} for shape meanings.

Using a computer program written in Sage, we verified that $H'$ has an injective $L'$-coloring $\varphi_{H'}$.
Notice $(G - H')^{(2)}$ is the same as $G^{(2)}-H'$ since every vertex in $H'$ has
at most one neighbor in $G-H'$.
Hence $\varphi_{H'}$ and $\varphi$ can be combined into an injective $L$-coloring of $G$, which is a contradiction.

The computer program verifies injective $L'$-colorability of $H'$ by a greedy coloring or by finding an Alon-Tarsi orientation (cf.~\cite{alon})
on the neighboring graph of $H'$. 
Both of these methods work only with list sizes and not with the actual contents of the lists.
The method used is denoted next to the labels in Figures~\ref{fig-2x}--\ref{fig-10x} by using AT for Alon-Tarsi and G for greedy.
The program we used can be obtained at \url{https://arxiv.org/abs/1611.03454}.
\end{proof}

Now we deal with the exceptions in $\mathcal{E}$.
Let $C_{2,2}^\star$ be the configuration depicted in Figure~\ref{fig-Bad}.

 \begin{claim}\label{cl:e}
If $(H,\md)\in \mathcal{E}$ and $(H,\md)$ appears in $G$, then it is $C_{2,2}^\star$.
\end{claim}
\begin{proof}
We use Claim~\ref{cl2edgeconnected} to show that
if any of the configurations $X_1$ and $X_2$ appear in $G$, then it must be $C_{2,2}^\star$.
The thick edges in Figure~\ref{fig-E} are edges $e_1$ and $e_2$ 
in Claim~\ref{cl2edgeconnected}.
Notice that $G-\{e_1,e_2\}$ is disconnected.

If neither $e_1$ nor $e_2$ is a bridge, then only one connected component contains precolored vertices
and the vertices of $e_1$ and $e_2$ are not precolored.
Also, notice that $X_1$ can be embedded in two different ways that are suggested by the dotted edges, but both cases contain the desired edges $e_1$ and $e_2$.

Suppose one of $e_1$ or $e_2$ is a bridge.
By Claim~\ref{cl:bridge}, the other is also a bridge and both are incident to
precolored vertices. This can happen only in $X_2$ and we obtain configuration $C_{2,2}^\star$.
\end{proof}

\begin{figure}[h!]
\[
\begin{array}{cc}
\vc{
\begin{tikzpicture}[scale=2]  
\node at (0.0,0.0) [listsize5,label=below left:$u$](x0){};
\node at (-0.35,0) [listsize5](x1){};
\node at (-.5,-0.5) [listsize4](x2){};
\node at (-0.35,-1) [listsize4](x3){};
\node at (0.0,-1) [listsize5](x4){};
\node at (0.0,-0.5) [listsize5,label=left:$v$](x5){};
\node at (0.35,0) [listsize5](x6){};
\node at (0.5,-0.5) [listsize4](x7){};
\node at (0.35,-1) [listsize4](x8){};
\node at (0,0.7) [listsize3](x15){};
\node at (-0.5,0.3) [listsize1](x9){};
\node at (0.5,0.3) [listsize1](x12){};

\node at (-0.35,-1.35) [listsize2](x20){};
\node at (0,-1.5) [listsize1](x21){};
\node at (0.35,-1.35) [listsize2](x22){};

\draw[realedges]  (x3)--(x20)  (x20)--(x21) (x21)--(x22) (x22)--(x8);
\draw[externaledges]  (x20)-- ++(270:0.25) (x21)-- ++(270:0.25)  (x22)-- ++(270:0.25);

\coordinate (x10) at (0,0.9);
\coordinate (x11) at (-0.6,-1);
\coordinate (x13) at (.75,-.5);
\coordinate (x14) at (.6,-1);
\draw (x5) node[degree2]{};
\draw[realedges]  (x0)--(x1) (x0)--(x5) (x0)--(x6) (x1)--(x2) (x2)--(x3) (x3)--(x4) (x4)--(x5) (x4)--(x8) (x6)--(x7) (x7)--(x8);
\draw[externaledges]   (x15)--(x10);
\draw[realedges] (x7) to[out=60,in=0,looseness=2] (x15) (x2) to[out=120,in=180,looseness=2] (x15); 
\draw[line width=1pt] 
 (x6)--node[pos=0.5,label=right:$e_2$]{}(x12)
 (x12)  -- ++(45:0.2) node[listsize1]{}
 (x12)  -- ++(135:0.2) node[listsize1]{}
 (x9)  -- ++(45:0.2) node[listsize1]{}
 (x9)  -- ++(135:0.2) node[listsize1]{}
 (x1)--node[pos=0.5,label=left:$e_1$]{}(x9);
\node at (-.25,-.7){ 6};  
\node at (.25,-.7){ 6};    
\node at (0,0.35){$F_o$};    
\node at (0,-1.25){6};  
\end{tikzpicture}
}
\\[10pt]
C_{2,2}^\star
\end{array}
\]
\caption{A non-reducible configuration $C_{2,2}^\star$ obtained from $C_{2,2}$.
Black vertices are in  $\mathcal{P}$ and the white vertices are in $G - \mathcal{P}$.
The vertex $v$ is a bad vertex and $F_o$ is the outer face.
The face $F_o$ is not actually drawn as the outer face in order to show correspondence with $X_2$ in 
Figure~\ref{fig-E}.
}\label{fig-Bad}
\end{figure}
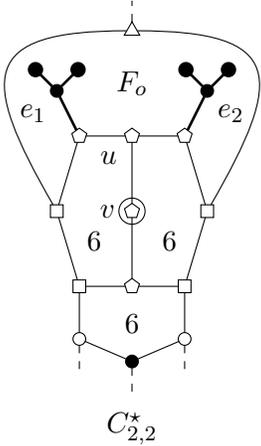


\subsection{Discharging Argument}

Recall that $G$ is a minimum counterexample to Theorem~\ref{thm:ICN:realresult}.
Hence $G$ is a plane graph with maximum degree at most 3 and girth at least 6.
In Section~\ref{secreducible} we showed that $G$ cannot contain any of the reducible configurations listed in Figures~\ref{fig-2x}--\ref{fig-10x}.
Moreover, $G$ contains at most two precolored paths, whose vertices are denoted by $\mathcal{P}$.  Recall $|\mathcal{P}|\leq6$.
Also, all precolored vertices are part of the outer face $F_o$.

For each vertex $v \in V(G) - \mathcal{P}$ and each face $f \in F(G) - \{F_o\}$, 
define initial charges $\mu_0 (v) = 2 \mathrm{deg} (v) - 6$ and $\mu_0 (f) = \ell (f) - 6$.
For precolored vertices $v$, define $\mu_0(v) = 0$ and for $F_o$ define $\mu_0 (F_o) = \ell (F_o) - 5 - |\mathcal{P}|$.
Let $\mathcal{P}_1$ be the precolored vertices of degree one. Recall that $|\mathcal{P}_1| = \frac{2}{3}|\mathcal{P}| \leq 4$.
By using Euler's formula~\eqref{eqeuler}
and
\[
\sum_{v \in V(G) } \deg(v) = 2|E| = \sum_{f \in F(G)} \ell(f),
\]
 we show that the sum of all charges is negative.
\begin{align}
|V(G)| + |F(G)|  &= |E(G)| + 2 \label{eqeuler}\\
4|E(G)|-6|V(G)| + 2|E(G)|-6|F(G)|  &= - 12 \nonumber\\
\sum_{v \in V(G)}(2\deg(v)-6) + \sum_{f \in F(G)}(\ell(f) - 6)  &= -12 \nonumber\\
\sum_{v \in V(G)}\mu_0(v) + \sum_{f \in F(G)} \mu_0(f)   -4|\mathcal{P}_1| + |\mathcal{P}|-1 &= -12 \nonumber\\
\sum_{v \in V(G)}\mu_0(v) + \sum_{f \in F(G)} \mu_0(f)  &\leq -1 \label{eqneg}
\end{align}
We sequentially  apply the following five discharging rules; see Figure~\ref{figRules} for an illustration.

\begin{flushleft} 
Let $v$ be a 2-vertex and let $f_1$ and $f_2$ be the faces incident to $v$.
\end{flushleft}
\begin{description}
	 
		\item{(R1)} If $\ell (f_1) = 6$ and $\ell (f_2) \geq 8$, $v$ pulls charge 2 from $f_2$.
		\item{(R2)} If $\ell (f_1) \geq 7$ and $\ell (f_2) \geq 7$, $v$ pulls charge 1 from $f_1$ and from $f_2$.
		\item{(R3)} If $\ell (f_1) = 6$ and $\ell (f_2) = 7$, $v$ pulls charge 1 from $f_2$.

\end{description}
\begin{figure}[h!]
\footnotesize
\centering
\stackunder[5pt]{\begin{tikzpicture}
\draw (0,0) -- (2,0) -- (2,2) -- (0,2) -- (0,0);
\draw (0,0) -- (2,0) -- (2,-2) -- (0,-2) -- (0,0);
\draw (1,-.2) node[anchor=south] {\textbullet};
\node at (1,1) {\stackanchor{$f_1$}{$\ell(f_1) = 6$}};
\node at (1,-1.5) {\stackanchor{$f_2$}{$\ell(f_2) \geq 8$}};
\draw[->,line width=0.3mm, black] (1,-1)--(1,-.05) node [pos=0.5, below, black, sloped] {+2};
\end{tikzpicture}}{(R1)}%
\hspace{1cm}%
\stackunder[5pt]{\begin{tikzpicture}
\draw (0,0) -- (2,0) -- (2,2) -- (0,2) -- (0,0);
\draw (0,0) -- (2,0) -- (2,-2) -- (0,-2) -- (0,0);
\draw (1,-.2) node[anchor=south] {\textbullet};
\node at (1,1.5) {\stackanchor{$f_1$}{$\ell(f_1) \geq 7$}};
\node at (1,-1.5) {\stackanchor{$f_2$}{$\ell(f_2) \geq 7$}};
\draw[->, line width=0.3mm, black] (1,1)--(1,.05) node [pos=0.5, above, black, sloped] {+1};
\draw[->,line width=0.3mm, black] (1,-1)--(1,-.05) node [pos=0.5, below, black, sloped] {+1};
\end{tikzpicture}}{(R2)}
\hspace{1cm}%
\stackunder[5pt]{\begin{tikzpicture}
\draw (0,0) -- (2,0) -- (2,2) -- (0,2) -- (0,0);
\draw (0,0) -- (2,0) -- (2,-2) -- (0,-2) -- (0,0);
\draw (1,-.2) node[anchor=south] {\textbullet};
\node at (1,1) {\stackanchor{$f_1$}{$\ell(f_1) = 6$}};
\node at (1,-1.5) {\stackanchor{$f_2$}{$\ell(f_2) =7$}};
\draw[->,line width=0.3mm, black] (1,-1)--(1,-.05) node [pos=0.5, below, black, sloped] {+1};
\end{tikzpicture}}{(R3)} \\ \vspace{0.5cm}
\stackunder[5pt]{\begin{tikzpicture}
\draw (0,0) -- (2,0) -- (2,2) -- (0,2) -- (0,0);
\draw (0,0) -- (2,0) -- (2,-2) -- (0,-2) -- (0,0);
\draw (0,-2) -- (-2,-2) -- (-2,2)--(0,2);
\draw (2,-2) -- (4,-2) -- (4,2)--(2,2);
\draw (1,-.2) node[anchor=south] {\textbullet};
\draw (0,-.2) node[anchor=south] {\textbullet};
\node at (-0.2,.25) {$v_1$};
\draw (2,-.2) node[anchor=south] {\textbullet};
\node at (2.25,0) {$v_2$};
\node at (1,1)  {\stackanchor{$f_1$}{$\ell(f_1) = 6$}};
\node at (1,-1.5)  {\stackanchor{$f_2$}{$\ell(f_2) = 7$}};
\node at (-1,0) {\stackanchor{$f_3$}{$\ell(f_3) \geq 7$}};
\node at (3,0) {$ $};
\node at (1,.3) {{\fontsize{3mm}{3mm}\selectfont $\mu_{3}(v) = -1$}};
\draw[->,line width=0.3mm, black] (-1,-1)--(1,0) node [pos=0.2, below, black, sloped] {+1};
\end{tikzpicture}}{(R4)}
\hspace{1cm}%
\stackunder[5pt]{\begin{tikzpicture}
\draw (0,0) -- (2,0) -- (2,2) -- (0,2) -- (0,0);
\draw (0,0) -- (2,0) -- (2,-2) -- (0,-2) -- (0,0);
\draw (0,-2) -- (-2,-2) -- (-2,2)--(0,2);
\draw (2,-2) -- (4,-2) -- (4,2)--(2,2);
\draw (1,-.2) node[anchor=south] {\textbullet};
\draw (0,-.2) node[anchor=south] {\textbullet};
\node at (-0.2,.25) {$v_1$};
\draw (2,-.2) node[anchor=south] {\textbullet};
\node at (2.25,.25) {$v_2$};
\node at (1,1)  {\stackanchor{$f_1$}{$\ell(f_1) = 6$}};
\node at (1,-1.5)  {\stackanchor{$f_2$}{$\ell(f_2) = 6$}};
\node at (-1,0) {\stackanchor{$f_3$}{$\ell(f_3) \geq 7$}};
\node at (3,0) {\stackanchor{$f_4$}{$\ell(f_4) \geq 7$}};
\node at (1,.3) {{\fontsize{3mm}{3mm}\selectfont $\mu_{3}(v) = -2$}};
\draw[->,line width=0.3mm, black] (-1,-1)--(1,0) node [pos=0.2, below, black, sloped] {+1};
\draw[->,line width=0.3mm, black] (3,-1)--(1,0) node [pos=0.2, below, black, sloped] {+1};
\end{tikzpicture}}{(R4)}

\caption{Rules (R1)--(R4).}\label{figRules}
\end{figure}
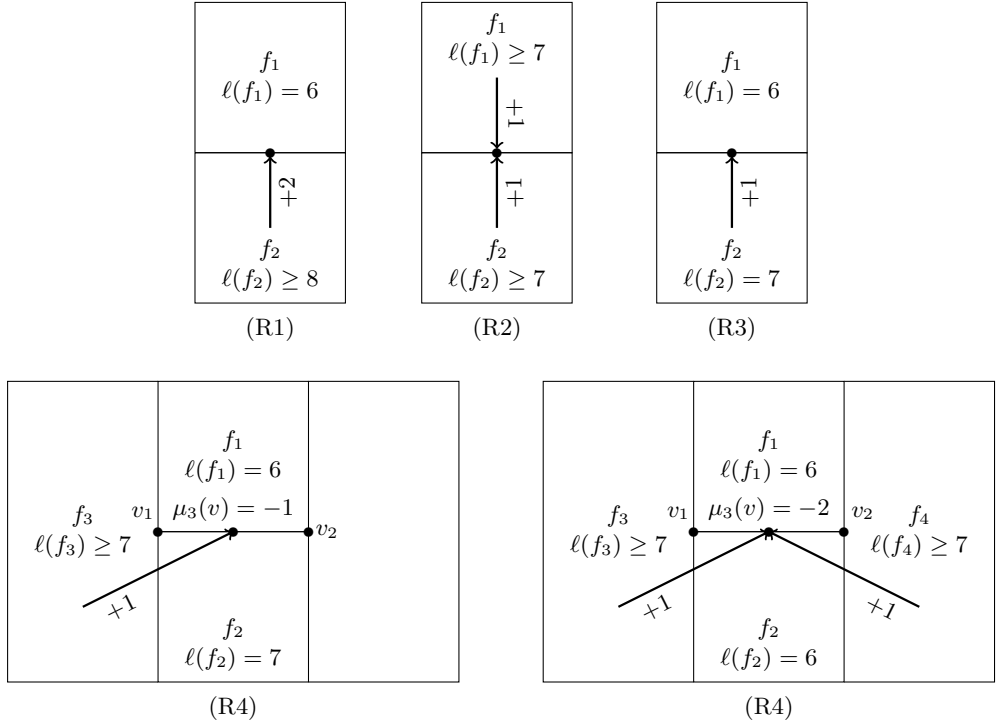

Let $\mu_i$ be the charge after applying rules (R1) to (R$i$) for $i \in \{3,4,5\}$.
Let $f_3$ be a face such that $v$ (defined above) is a nearby vertex of $f_3$. Recall $F_o$ is the outer face.
	
\begin{description}

		\item{(R4)} { If $\mu_{3} (v) < 0$ and $\ell(f_3) \geq 7$, $v$ pulls charge 1 from $f_3$.}

		\item{(R5)} { If $\mu_{4} (v) < 0$, $v$ pulls charge 2 from $F_o$.}

\end{description}

Notice rules (R4) and (R5) apply only if $v$ is incident to at least one 6-face. 
A 2-vertex $u$ is \emph{needy} if $\mu_{3}(u) < 0$.
A 2-vertex $u$ is \emph{bad} if $\mu_{4}(u) < 0$.

Now we  show that all vertices and faces in the minimal counterexample have nonnegative final charge $\mu_5$, providing our contradiction with \eqref{eqneg}. 
Note that 3-vertices and 6-faces begin with charge 0 and never lose any charge due to the discharging rules, thus they end with nonnegative charge. 
By Claim~\ref{claimprecolored}, none of the precolored vertices have degree 2.
Since precolored vertices begin with charge 0 and they are not affected by any of the discharging rules, their final charge is 0.

By Claim~\ref{cl:leaf}, the minimum vertex degree of $G-\mathcal{P}$ is 2.

\begin{figure}[h]
\[
\begin{array}{cc}
\begin{tikzpicture}[scale=2]  
\node at (0.0,0.0) [listsize5](x0){};
\node at (0.1,-0.22) [listsize3](x1){};
\node at (.5,-0.2) [listsize3](x2){};
\node at (0.55,0.15) [listsize3](x3){};
\node at (0.37,0.33) [listsize5](x4){};
\node at (0.18,0.16) [listsize5](x5){};
\node at (-0.22,0.13) [listsize3](x6){};
\node at (-.19,0.43) [listsize2](x7){};
\node at (.04,0.61) [listsize2](x8){};
\node at (0.32,0.57) [listsize3](x9){};
\node at (0.81,0.28) [listsize2](x10){};
\node at (0.81,0.6) [listsize1](x11){};
\node at (0.55,0.75) [listsize2](x12){};
\node at (-0.05,-0.45) [listsize2](x13){};
\node at (-0.35,-0.34) [listsize1](x14){};
\node at (-0.45,-0.05) [listsize2](x15){};
\coordinate (x16) at (0.7,-0.35);
\coordinate (x17) at (-0.35,0.55);
\coordinate (x18) at (-0.03,0.8);
\coordinate (x19) at (1.0,0.22);
\coordinate (x20) at (1.0,0.7);
\coordinate (x21) at (0.55,1);
\coordinate (x22) at (0,-0.65);
\coordinate (x23) at (-0.49,-0.5);
\coordinate (x24) at (-0.65,-0.05);
\draw (x5) node[degree2]{};
\draw[realedges]  (x0)--(x1) (x0)--(x5) (x0)--(x6) (x1)--(x2)
(x1)--(x13) (x2)--(x3) (x3)--(x4) (x3)--(x10) (x4)--(x5) (x4)--(x9)
(x6)--(x7) (x6)--(x15) (x7)--(x8) (x8)--(x9) (x9)--(x12) (x10)--(x11)
(x11)--(x12) (x13)--(x14) (x14)--(x15);
\draw[externaledges]  (x2)--(x16) (x7)--(x17) (x8)--(x18) (x10)--(x19)
(x11)--(x20) (x12)--(x21) (x13)--(x22) (x14)--(x23) (x15)--(x24);
\node at (0.25,0.33)[label=left:7]{}; 
\node at (.35,0){6}; 
\node at (.6,.45){6}; 
\node at (-.155,-.155){6};   
\end{tikzpicture}
&
\begin{tikzpicture}[scale=2]  
\node at (0.0,0.0) [listsize5](x0){};
\node at (-0.35,0) [listsize3](x1){};
\node at (-.5,-0.5) [listsize2](x2){};
\node at (-0.35,-1) [listsize2](x3){};
\node at (0.0,-1) [listsize3](x4){};
\node at (0.0,-0.5) [listsize5](x5){};
\node at (0.35,0) [listsize3](x6){};
\node at (0.5,-0.5) [listsize2](x7){};
\node at (0.35,-1) [listsize2](x8){};
\node at (-0.35,0.35) [listsize2](x9){};
\node at (0,0.5) [listsize1](x10){};
\node at (0.35,0.35) [listsize2](x11){};
\coordinate (x12) at (-0.75,-.5);
\coordinate (x13) at (-0.6,-1);
\coordinate (x14) at (.75,-.5);
\coordinate (x15) at (.6,-1);
\coordinate (x16) at (-0.35,0.6);
\coordinate (x17) at (0,0.75);
\coordinate (x18) at (0.35,0.6);
\draw (x5) node[degree2]{};
\draw[realedges]  (x0)--(x1) (x0)--(x5) (x0)--(x6) (x1)--(x2) (x2)--(x3) (x3)--(x4) (x4)--(x5);
\draw[realedges]  (x4)--(x8) (x6)--(x7) (x7)--(x8) (x1)--(x9) (x9)--(x10) (x10)--(x11) (x11)--(x6);
\draw[externaledges]  (x2)--(x12) (x3)--(x13) (x7)--(x14) (x8)--(x15) (x9)--(x16) (x10)--(x17) (x11)--(x18);
\node at (0.25,-.5){6};
\node at (-0.25,-.5){6}; 
\node at (0,.25){6};  
\end{tikzpicture}
\\
C_{2,1}, \hspace{3px} $\alon$  & C_{2,2}, \hspace{3px} $\greedy$  \\
\end{array}
\]
\caption{Configurations around a 2-vertex, where $C_{2,1}$ is reducible and $C_{2,2}$ is reducible except for the special case  $C_{2,2}^\star$ depicted in Figure~\ref{fig-Bad}.}\label{fig-2x}
\end{figure}
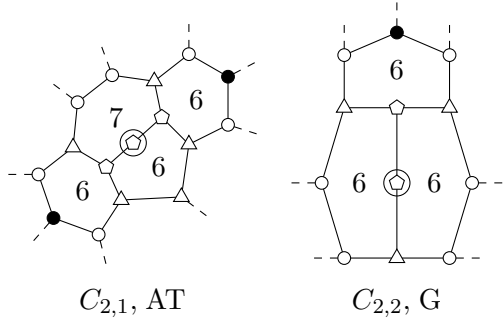

Let $v$ be a 2-vertex with neighbors $v_1$ and $v_2$.
Note the reducible configurations in Figure~\ref{fig-2x}.
The minimum distance between 2-vertices must be at least 4 by Claim \ref{claim1}. 
Hence $v_1$ and $v_2$ must be 3-vertices.  
Let $f_1$ and $f_2$ be the faces incident to $v$. 
Since $v$ begins with charge $\mu_0 (v) = -2$, we must show that it receives charge at least 2 during discharging.  
If $\ell (f_1) \geq 7$ and $\ell (f_2) \geq 7$, then $v$ receives charge at least 1 from each face by (R2), resulting in $\mu_3(v)=\mu_5(v)=0$. 
Assume then that $\ell (f_1) = 6$. 
\begin{itemize}
\item If $\ell (f_2) \geq 8$, then $v$ receives charge 2 from $f_2$ by (R1), resulting in $\mu_3(v)=\mu_5(v)=0$. 
\item If $\ell (f_2) = 7$, then $v$ receives charge 1 from $f_2$ by (R3). This gives $\mu_3(v) = -1$
and $v$ is a needy vertex.
Let $f_3$ and $f_4$ be the faces incident to $v_1$ and $v_2$, respectively, that are not incident to $v$. 
If $\max\{\ell(f_3),\ell(f_4)\} \geq 7$, then $v$ receives an additional 1 charge by (R4). This results in $\mu_5(v) \geq 0$.
The case $\ell (f_3) = \ell (f_4) = 6$ cannot happen, otherwise $G$ would contain the reducible configuration $C_{2,1}$.  
\item  If $\ell (f_2) = 6$, then $\mu_3(v) = -2$.  Let $f_3$ and $f_4$ be the faces incident to $v_1$ and $v_2$, respectively, that are not incident to $v$.
If $\ell (f_3) \geq 7$ and $\ell (f_4) \geq 7$, then $v$ receives charge 1 from each of $f_3$ and $f_4$.
If $\ell (f_3) = 6$ or $\ell (f_4) = 6$, $G$ cannot contain $C_{2,2}$ and hence must contain $C_{2,2}^{\star}$ from Figure~\ref{fig-Bad} by Claim~\ref{cl:e}.
In this case, $v$ is a bad vertex and will thus receive charge 2 from $F_o$.
Observe that there is at most one bad vertex $v$ in $G$.
Indeed, a neighbor $u$ of a bad vertex is in $F_o$ and both neighbors of $u$ in $F_o$ have neighbors in $\mathcal{P}$. 
Since $G$ has girth $6$, there can be at most one such vertex $u$; see Figure~\ref{fig-Bad}.
\end{itemize}
Hence, every 2-vertex has a nonnegative final charge. 

Let $f$ be a face.  Due to the minimum distance between 2-vertices, $f$ can be incident to at most $\left\lfloor \frac{\ell(f)}4 \right\rfloor$ 2-vertices.   The minimum distance restriction also limits the number of 2-vertices that are nearby $f$, as shown in Table~\ref{table:IN}.  Note that this table only takes the minimum distance restriction into consideration and not any other reducible configurations.  For each of the possible configurations in the table, we argue that $f$ will either be part of a reducible configuration or have nonnegative final charge. The case $\ell (f) \geq 11$ will be handled separately.

If $|\mathcal{P}| \geq 1$, then $\ell(F_o) \geq 12$ and $F_o$ is included in the case $\ell (f) \geq 11$.
If there are no precolored vertices, $F_o$ is treated as any other face of $G$.

\begin{table}
\begin{center}
	\begin{tabular}{|c|c|c|c|c|} \hline
		$\ell (f)$ & 7 & 8 & 9 & 10\\ \hline
		$\left|I(f)\right|$ & \, 1 \; 0 \, & \, 2 \; 1 \; 0 \, & \, 2 \; 1 \; 0 \, & \, 2 \; 1 \; 0 \,\\ \hline
		$\max \left|N(f)\right|$ & \, 1 \; 3 \, & \, 0 \; 2 \; 4 \, & \, 0 \; 2 \; 4 \, & \, 1 \; 3 \; 5 \,\\ \hline
	\end{tabular}
\end{center}
\caption{Maximum number of nearby vertices of $f$ when $\ell(f)$ and $|I(f)|$ are fixed.}\label{table:IN}
\end{table}

We distinguish the case based on $\ell(f)$.
We also include charts to summarize the results based on the length of the face. Each cell will contain the argument for why the configuration is reducible, has a nonnegative final charge, denoted EC, or fails to meet the distance requirement, denoted DR.

\begin{itemize}

\item 
 Suppose that $\ell (f) = 6$.
 We have $\mu_0 (f) = 0 = \mu_5 (f)$ since $f$ does not send or receive any charge.
 Hence, if $\ell (f) = 6$, then $f$ has nonnegative final charge.

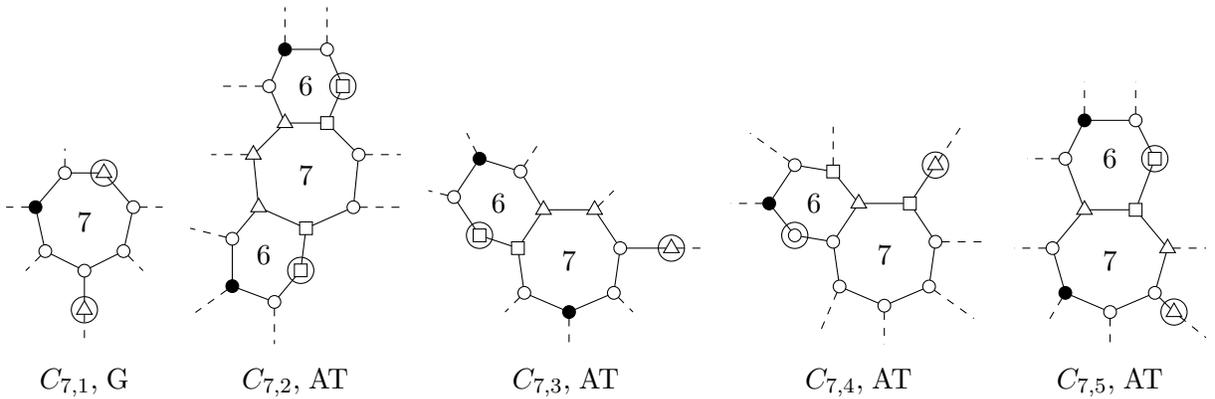
\begin{figure}[h]
\[
\begin{array}{ccccc}
\begin{tikzpicture}[scale=1.3]  
\node at (0.0,0.0) [listsize2](x0){};
\node at (0.4,0.2) [listsize2](x1){};
\node at (0.5,0.65) [listsize2](x2){};
\node at (0.2,1) [listsize3](x3){};
\node at (-0.2,1) [listsize2](x4){};
\node at (-.5,0.65) [listsize1](x5){};
\node at (-0.4,0.2) [listsize2](x6){};
\node at (-0,-0.4) [listsize3](x7){};
\coordinate (x8) at (0.6,0);
\coordinate (x9) at (.8,0.65);
\coordinate (x10) at (-0.2,1.25);
\coordinate (x11) at (-0.8,.65);
\coordinate (x12) at (-0.6,-0);
\coordinate (x13) at (0,-0.75);
\draw (x3) node[degree2]{} (x7) node[degree2]{};
\draw[realedges]  (x0)--(x1) (x0)--(x6) (x0)--(x7) (x1)--(x2) (x2)--(x3) (x3)--(x4) (x4)--(x5) (x5)--(x6);
\draw[externaledges]  (x1)--(x8) (x2)--(x9) (x4)--(x10) (x5)--(x11) (x6)--(x12) (x7)--(x13);
\node at (0.0,0.5){ 7 };   
\end{tikzpicture}
&
\begin{tikzpicture}[scale=1.4]  
\node at (0.1,0.0) [listsize4](x0){};
\node at (0.4,-0.3) [listsize2](x1){};
\node at (0.35,-0.8) [listsize2](x2){};
\node at (-.1,-1) [listsize4](x3){};
\node at (-0.55,-.8) [listsize3](x4){};
\node at (-0.6,-0.3) [listsize3](x5){};
\node at (-0.3,0) [listsize3](x6){};
\node at (-.15,-1.4) [listsize4](x7){};
\node at (.25,0.35) [listsize4](x8){};
\node at (.1,.7) [listsize2](x9){};
\node at (-.3,.7) [listsize1](x10){};
\node at (-0.45,0.35) [listsize2](x11){};
\node at (-0.8,-1.1) [listsize2](x12){};
\node at (-0.8,-1.55) [listsize1](x13){};
\node at (-0.4,-1.7) [listsize2](x14){};
\coordinate (x15) at (0.8,-0.3);
\coordinate (x16) at (0.8,-0.8);
\coordinate (x17) at (-1,-0.3);
\coordinate (x18) at (0.1,1.1);
\coordinate (x19) at (-0.3,1.1);
\coordinate (x20) at (-.89,0.35);
\coordinate (x21) at (-1.2,-1);
\coordinate (x22) at (-1.15,-1.8);
\coordinate (x23) at (-.4,-2.1);
\draw (x7) node[degree2]{} (x8) node[degree2]{};
\draw[realedges]  (x0)--(x1) (x0)--(x6) (x0)--(x8) (x1)--(x2) (x2)--(x3) (x3)--(x4) (x3)--(x7) (x4)--(x5) (x4)--(x12) (x5)--(x6) (x6)--(x11) (x7)--(x14) (x8)--(x9) (x9)--(x10) (x10)--(x11) (x12)--(x13) (x13)--(x14);
\draw[externaledges]  (x1)--(x15) (x2)--(x16) (x5)--(x17) (x9)--(x18) (x10)--(x19) (x11)--(x20) (x12)--(x21) (x13)--(x22) (x14)--(x23);
\node at (-0.1,-0.5){$7$}; 
\node at (-0.1,0.35){$6$};
\node at (-0.5,-1.25){$6$};  
\end{tikzpicture}
&
\begin{tikzpicture}[scale=1.7]  
\node at (0.0,0.0) [listsize1](x0){};
\node at (.35,0.15) [listsize2](x1){};
\node at (.4,0.5) [listsize2](x2){};
\node at (0.2,0.8) [listsize3](x3){};
\node at (-0.2,0.8) [listsize3](x4){};
\node at (-0.4,0.5) [listsize4](x5){};
\node at (-.35,.15) [listsize2](x6){};
\node at (.8,.5) [listsize3](x7){};
\node at (-0.7,0.6) [listsize4](x8){};
\node at (-.9,.9) [listsize2](x9){};
\node at (-0.7,1.2) [listsize1](x10){};
\node at (-0.38,1.1) [listsize2](x11){};
\coordinate (x12) at (0,-0.25);
\coordinate (x13) at (0.5,-0);
\coordinate (x14) at (0.35,.95);
\coordinate (x15) at (-.5,0);
\coordinate (x16) at (1.05,0.5);
\coordinate (x17) at (-1.1,.95);
\coordinate (x18) at (-.8,1.4);
\coordinate (x19) at (-.25,1.3);
\draw (x7) node[degree2]{} (x8) node[degree2]{};
\draw[realedges]  (x0)--(x1) (x0)--(x6) (x1)--(x2) (x2)--(x3) (x2)--(x7) (x3)--(x4) (x4)--(x5) (x4)--(x11) (x5)--(x6) (x5)--(x8) (x8)--(x9) (x9)--(x10) (x10)--(x11);
\draw[externaledges]  (x0)--(x12) (x1)--(x13) (x3)--(x14) (x6)--(x15) (x7)--(x16) (x9)--(x17) (x10)--(x18) (x11)--(x19);
\node at (-0,0.4){ 7 }; 
\node at (-.55,.85){6};   
\end{tikzpicture}
&
\begin{tikzpicture}[scale=1.7]  
\node at (0.2,0.8) [listsize4](x0){};
\node at (0.4,0.5) [listsize2](x1){};
\node at (0.35,0.15) [listsize2](x2){};
\node at (0,0) [listsize2](x3){};
\node at (-0.35,0.15) [listsize2](x4){};
\node at (-0.4,0.5) [listsize2](x5){};
\node at (-0.2,0.8) [listsize3](x6){};
\node at (0.4,1.1) [listsize3](x7){};
\node at (-0.4,1.05) [listsize4](x8){};
\node at (-0.7,1.1)[listsize2](x9){};
\node at (-0.9,.8) [listsize1](x10){};
\node at (-0.7,0.55) [listsize2](x11){};
\coordinate (x12) at (.8,0.5);
\coordinate (x13) at (0.7,-.1);
\coordinate (x14) at (0,-0.3);
\coordinate (x15) at (-0.5,-0.2);
\coordinate (x16) at (.6,1.4);
\coordinate (x17) at (-0.4,1.4);
\coordinate (x18) at (-1.05,1.35);
\coordinate (x19) at (-1.2,.8);
\draw (x7) node[degree2]{} (x11) node[degree2]{};
\draw[realedges]  (x0)--(x1) (x0)--(x6) (x0)--(x7) (x1)--(x2) (x2)--(x3) (x3)--(x4) (x4)--(x5) (x5)--(x6) (x5)--(x11) (x6)--(x8) (x8)--(x9) (x9)--(x10) (x10)--(x11);
\draw[externaledges]  (x1)--(x12) (x2)--(x13) (x3)--(x14) (x4)--(x15) (x7)--(x16) (x8)--(x17) (x9)--(x18) (x10)--(x19);
\node at (0,0.4){7 }; 
\node at (-.55,.8){6};   
\end{tikzpicture}
&
\begin{tikzpicture}[scale=1.7]  
\node at (0,0.0) [listsize2](x0){};
\node at (-0.35,0.15) [listsize1](x1){};
\node at (-0.45,0.5) [listsize2](x2){};
\node at (-0.2,0.8) [listsize3](x3){};
\node at (0.2,0.8) [listsize4](x4){};
\node at (.45,0.5) [listsize3](x5){};
\node at (0.35,0.15) [listsize2](x6){};
\node at (0.5,0) [listsize3](x7){};
\node at (0.35,1.2) [listsize4](x8){};
\node at (0.2,1.5) [listsize2](x9){};
\node at (-0.2,1.5) [listsize1](x10){};
\node at (-0.35,1.2) [listsize2](x11){};
\coordinate (x12) at (0.0,-0.25);
\coordinate (x13) at (-0.5,-0.1);
\coordinate (x14) at (-0.75,0.5);
\coordinate (x15) at (.75,0.5);
\coordinate (x16) at (0.75,-0.2);
\coordinate (x17) at (0.2,1.8);
\coordinate (x18) at (-0.2,1.8);
\coordinate (x19) at (-0.65,1.2);
\draw (x7) node[degree2]{} (x8) node[degree2]{};
\draw[realedges]  (x0)--(x1) (x0)--(x6) (x1)--(x2) (x2)--(x3) (x3)--(x4) (x3)--(x11) (x4)--(x5) (x4)--(x8) (x5)--(x6) (x6)--(x7) (x8)--(x9) (x9)--(x10) (x10)--(x11);
\draw[externaledges]  (x0)--(x12) (x1)--(x13) (x2)--(x14) (x5)--(x15) (x7)--(x16) (x9)--(x17) (x10)--(x18) (x11)--(x19);
\node at (0,0.4){ 7}; 
\node at (0,1.2){6};   
\end{tikzpicture}
\\
C_{7,1},\hspace{3px}  $\greedy$ & C_{7,2}, \hspace{3px} $\alon$  & C_{7,3}, \hspace{3px} $\alon$  & C_{7,4}, \hspace{3px} $\alon$  & C_{7,5}, \hspace{3px} $\alon$\\
\end{array}
\]
\caption{Reducible configurations around a 7-face.}\label{fig-7x}
\end{figure}
\item 
 Suppose that $\ell (f) = 7$; $\mu_0(f) = 1$,
 and note the reducible configurations in Figure~\ref{fig-7x}.
 Since $f$ begins with charge $\mu_0 (f) = 1$, we must show that it loses charge at most 1.  
 
 Suppose that $f$ is incident to one 2-vertex $v$.
 By configuration {$C_{7,1}$}, $f$ cannot have a nearby 2-vertex. 
 If $f$ has no nearby 2-vertices, then it only loses charge 1 to $v$ by (R2) or (R3), thus  $\mu_5(f) \geq 0$.
 
Suppose that $f$ is not incident to any 2-vertex.   
We show that at most one nearby 2-vertex of $f$  is needy.
Suppose for contradiction that $v_1$ and $v_2$ are two nearby vertices and both are needy.
Since they are needy, each must be in a 6-face. 
If they are distance 5 apart, $G$ contains {$C_{7,2}$} or {$C_{7,3}$}, which is a contradiction.
If they are distance 4 apart, $G$ contains {$C_{7,4}$} or {$C_{7,5}$}, which is a contradiction.
Hence $f$ is incident to at most one needy vertex.
Therefore, (R4) applies to $f$  at most once.
Hence, if $\ell (f) = 7$, then $f$ has nonnegative final charge.

 \begin{center}	
\begin{tabular}{|c|c|c|c|} \hline
\multicolumn{4}{|c|}{{\bf $\ell(f) = 7$}} \\
\hline
		$\left|I(f)\right| / \left|N(f)\right|$ & 0 & 1 & 2   \\
		 \hline
		0 & EC & EC & $C_{7,2},C_{7,3},C_{7,4},C_{7,5}$  \\ \hline
		1 & EC & $C_{7,1}$ & DR   \\ \hline 
	\end{tabular}
\end{center}

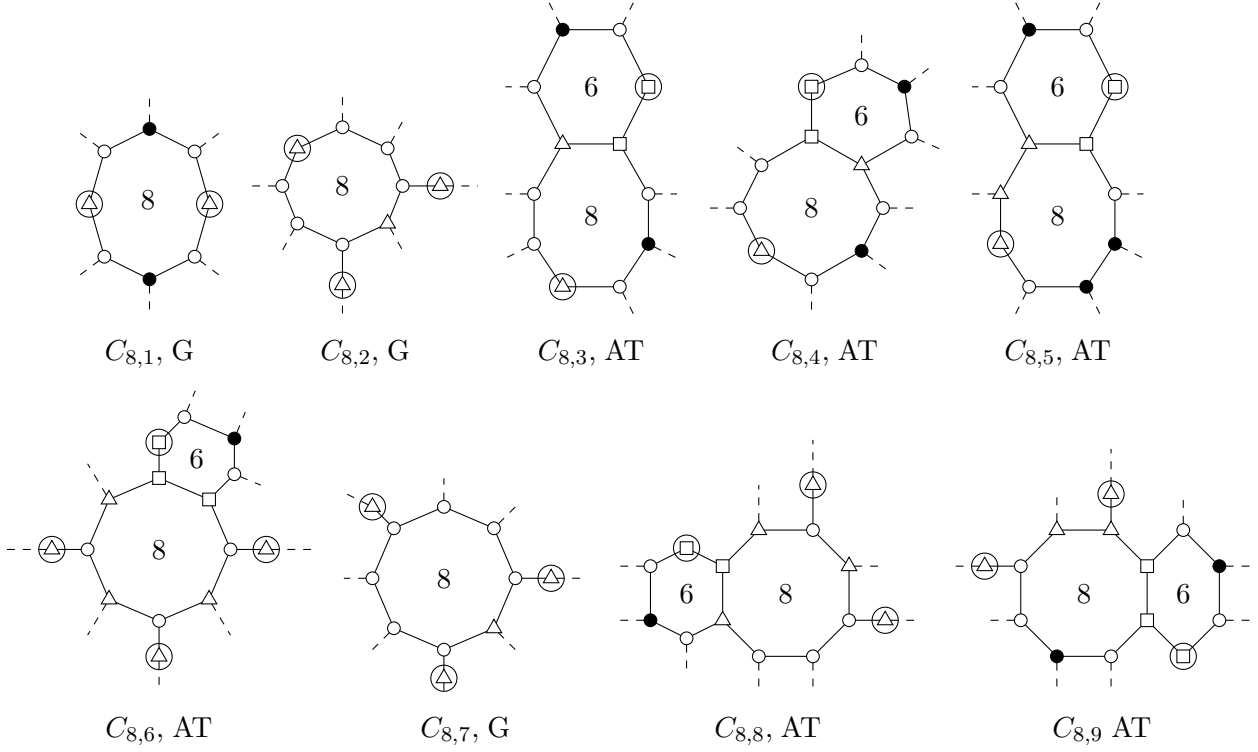
\begin{figure}
\[
\begin{array}{ccccc}	
\begin{tikzpicture}[scale=2]  
\node at (0,0) [listsize3](x0){};
\node at (0.1,-0.35) [listsize2](x1){};
\node at (0.4,-0.5) [listsize1](x2){};
\node at (0.7,-0.35) [listsize2](x3){};
\node at (0.8,0.0) [listsize3](x4){};
\node at (0.7,0.35) [listsize2](x5){};
\node at (0.4,0.5) [listsize1](x6){};
\node at (0.1,0.35) [listsize2](x7){};
\coordinate (x8) at (-0.1,-0.5);
\coordinate (x9) at (0.4,-0.74);
\coordinate (x10) at (0.9,-0.5);
\coordinate (x11) at (.9,0.5);
\coordinate (x12) at (0.4,0.74);
\coordinate (x13) at (-0.1,0.5);
\draw (x0) node[degree2]{} (x4) node[degree2]{};
\draw[realedges]  (x0)--(x1) (x0)--(x7) (x1)--(x2) (x2)--(x3) (x3)--(x4) (x4)--(x5) (x5)--(x6) (x6)--(x7);
\draw[externaledges]  (x1)--(x8) (x2)--(x9) (x3)--(x10) (x5)--(x11) (x6)--(x12) (x7)--(x13);
\node at (0.39,0.03){ 8 };   
\end{tikzpicture}
&
\begin{tikzpicture}[scale=2]  
\node at (0,0.0) [listsize2](x0){};
\node at (-0.1,0.25) [listsize2](x1){};
\node at (-0.4,0.39) [listsize2](x2){};
\node at (-0.7,0.25) [listsize3](x3){};
\node at (-0.8,0.0) [listsize2](x4){};
\node at (-0.7,-0.25) [listsize2](x5){};
\node at (-0.4,-0.39) [listsize2](x6){};
\node at (-0.1,-0.25) [listsize3](x7){};
\node at (0.25,-0.0) [listsize3](x8){};
\node at (-0.4,-0.66) [listsize3](x9){};
\coordinate (x10) at (-0.0,0.43);
\coordinate (x11) at (-0.4,0.6);
\coordinate (x12) at (-1.0,0);
\coordinate (x13) at (-0.8,-0.43);
\coordinate (x14) at (0.0,-0.43);
\coordinate (x15) at (0.5,-0.0);
\coordinate (x16) at (-0.4,-0.86);
\draw (x3) node[degree2]{} (x8) node[degree2]{} (x9) node[degree2]{};
\draw[realedges]  (x0)--(x1) (x0)--(x7) (x0)--(x8) (x1)--(x2) (x2)--(x3) (x3)--(x4) (x4)--(x5) (x5)--(x6) (x6)--(x7) (x6)--(x9);
\draw[externaledges]  (x1)--(x10) (x2)--(x11) (x4)--(x12) (x5)--(x13) (x7)--(x14) (x8)--(x15) (x9)--(x16);
\node at (-0.4,0.0){8 };   
\end{tikzpicture}
&

\begin{tikzpicture}[scale=1.9]  
\node at (0.0,0.0) [listsize4](x0){};
\node at (-0.4,0) [listsize3](x1){};
\node at (-0.6,-0.35) [listsize2](x2){};
\node at (-0.6,-0.7) [listsize2](x3){};
\node at (-0.4,-1) [listsize3](x4){};
\node at (-0.0,-1) [listsize2](x5){};
\node at (0.2,-0.7) [listsize1](x6){};
\node at (0.2,-0.35) [listsize2](x7){};
\node at (0.2,0.4) [listsize4](x8){};
\node at (0,0.8) [listsize2](x9){};
\node at (-0.4,0.8) [listsize1](x10){};
\node at (-0.6,0.4) [listsize2](x11){};
\coordinate (x12) at (-.8,-0.35);
\coordinate (x13) at (-0.8,-0.8);
\coordinate (x14) at (0.1,-1.2);
\coordinate (x15) at (0.4,-0.8);
\coordinate (x16) at (0.4,-0.35);
\coordinate (x17) at (0.1,1);
\coordinate (x18) at (-0.5,1);
\coordinate (x19) at (-0.8,0.4);
\draw (x4) node[degree2]{} (x8) node[degree2]{};
\draw[realedges]  (x0)--(x1) (x0)--(x7) (x0)--(x8) (x1)--(x2) (x1)--(x11) (x2)--(x3) (x3)--(x4) (x4)--(x5) (x5)--(x6) (x6)--(x7) (x8)--(x9) (x9)--(x10) (x10)--(x11);
\draw[externaledges]  (x2)--(x12) (x3)--(x13) (x5)--(x14) (x6)--(x15) (x7)--(x16) (x9)--(x17) (x10)--(x18) (x11)--(x19);
\node at (-0.2,-0.5){ 8 }; 
\node at (-0.2,.4){6};  
\end{tikzpicture}

&
\begin{tikzpicture}[scale=1.9]  
\node at (0.0,0.0) [listsize4](x0){};
\node at (0.35,-0.2) [listsize3](x1){};
\node at (0.5,-0.5) [listsize2](x2){};
\node at (0.35,-0.8) [listsize1](x3){};
\node at (0,-1) [listsize2](x4){};
\node at (-0.35,-0.8) [listsize3](x5){};
\node at (-0.5,-0.5) [listsize2](x6){};
\node at (-0.35,-0.2) [listsize2](x7){};
\node at (0,0.35) [listsize4](x8){};
\node at (0.35,0.5) [listsize2](x9){};
\node at (0.65,0.35) [listsize1](x10){};
\node at (0.7,-0) [listsize2](x11){};
\coordinate (x12) at (0.75,-0.5);
\coordinate (x13) at (0.55,-.95);
\coordinate (x14) at (-0,-1.25);
\coordinate (x15) at (-0.75,-0.5);
\coordinate (x16) at (-0.55,-.05);
\coordinate (x17) at (0.35,0.72);
\coordinate (x18) at (0.85,0.5);
\coordinate (x19) at (0.93,-0.1);
\draw (x5) node[degree2]{} (x8) node[degree2]{};
\draw[realedges]  (x0)--(x1) (x0)--(x7) (x0)--(x8) (x1)--(x2) (x1)--(x11) (x2)--(x3) (x3)--(x4) (x4)--(x5) (x5)--(x6) (x6)--(x7) (x8)--(x9) (x9)--(x10) (x10)--(x11);
\draw[externaledges]  (x2)--(x12) (x3)--(x13) (x4)--(x14) (x6)--(x15) (x7)--(x16) (x9)--(x17) (x10)--(x18) (x11)--(x19);
\node at (0.0,-0.5){ 8 };
\node at (0.35,.15){6};   
\end{tikzpicture}
&
\begin{tikzpicture}[scale=1.9]  
\node at (0.0,0.0) [listsize4](x0){};
\node at (-0.4,0) [listsize3](x1){};
\node at (-0.6,-0.35) [listsize3](x2){};
\node at (-0.6,-0.7) [listsize3](x3){};
\node at (-0.4,-1) [listsize2](x4){};
\node at (-0.0,-1) [listsize1](x5){};
\node at (0.2,-0.7) [listsize1](x6){};
\node at (0.2,-0.35) [listsize2](x7){};
\node at (0.2,0.4) [listsize4](x8){};
\node at (0,0.8) [listsize2](x9){};
\node at (-0.4,0.8) [listsize1](x10){};
\node at (-0.6,0.4) [listsize2](x11){};
\coordinate (x12) at (-.8,-0.35);
\coordinate (x13) at (-0.5,-1.2);
\coordinate (x14) at (0.1,-1.2);
\coordinate (x15) at (0.4,-0.8);
\coordinate (x16) at (0.4,-0.35);
\coordinate (x17) at (0.1,1);
\coordinate (x18) at (-0.5,1);
\coordinate (x19) at (-0.8,0.4);
\draw (x3) node[degree2]{} (x8) node[degree2]{};
\draw[realedges]  (x0)--(x1) (x0)--(x7) (x0)--(x8) (x1)--(x2) (x1)--(x11) (x2)--(x3) (x3)--(x4) (x4)--(x5) (x5)--(x6) (x6)--(x7) (x8)--(x9) (x9)--(x10) (x10)--(x11);
\draw[externaledges]  (x2)--(x12) (x4)--(x13) (x5)--(x14) (x6)--(x15) (x7)--(x16) (x9)--(x17) (x10)--(x18) (x11)--(x19);
\node at (-0.2,-0.5){ 8 }; 
\node at (-0.2,.4){6};  
\end{tikzpicture}
\\
C_{8,1}, \hspace{3px} $\greedy$  & C_{8,2},  \hspace{3px} $\greedy$  & C_{8,3},  \hspace{3px} $\alon$  & C_{8,4},  \hspace{3px} $\alon$ &C_{8,5},  \hspace{3px} $\alon$  \\
\end{array}
\]
\[
\begin{array}{cccc}
\begin{tikzpicture}[scale=1.9]  
\node at (0,0.0) [listsize4](x0){};
\node at (0.35,-0.15) [listsize4](x1){};
\node at (0.5,-0.5) [listsize2](x2){};
\node at (0.35,-0.85) [listsize3](x3){};
\node at (0,-1.0) [listsize2](x4){};
\node at (-0.35,-0.85) [listsize3](x5){};
\node at (-0.5,-0.5) [listsize2](x6){};
\node at (-0.35,-0.15) [listsize3](x7){};
\node at (0.75,-0.5) [listsize3](x8){};
\node at (0,-1.25) [listsize3](x9){};
\node at (-0.75,-0.5) [listsize3](x10){};

\draw[realedges]  (x1)  ++(45:0.25) node[listsize2](x20){}  ++ (90:0.25) node[listsize1](x21){};
\draw[realedges]  (x0)  ++(90:0.25) node[listsize4](x11){}  ++(45:0.25) node[listsize2](x22){}; 
\draw[externaledges]  (x20) -- ++(-22.5:0.2)  (x21) -- ++(67.5:0.2) (x22) -- ++(67.5:0.2) ;
\coordinate (x13) at (0.5,-1.1);
\coordinate (x14) at (-0.5,-1.1);
\coordinate (x15) at (-0.5,0.1);
\coordinate (x16) at (1.1,-0.5);
\coordinate (x17) at (0,-1.5);
\coordinate (x18) at (-1.1,-0.5);
\draw (x8) node[degree2]{} (x9) node[degree2]{} (x10) node[degree2]{} (x11) node[degree2]{};
\draw[realedges]  (x0)--(x1) (x0)--(x7)  (x1)--(x2) (x2)--(x3) (x2)--(x8) (x3)--(x4) (x4)--(x5) (x4)--(x9) (x5)--(x6) (x6)--(x7) (x6)--(x10) (x1)--(x20) (x20)--(x21) (x0)--(x11) (x11)--(x22) (x22)--(x21);
\draw[externaledges]  (x3)--(x13) (x5)--(x14) (x7)--(x15) (x8)--(x16) (x9)--(x17) (x10)--(x18);
\node at (0,-0.5){8};   
\node at ($(x11)!0.5!(x20)$){6};   
\end{tikzpicture}
&
\begin{tikzpicture}[scale=1.9]  
\node at (0.0,0.0) [listsize2](x0){};
\node at (-0.15,0.35) [listsize2](x1){};
\node at (-0.5,0.5) [listsize2](x2){};
\node at (-0.85,0.35) [listsize2](x3){};
\node at (-1,-0) [listsize2](x4){};
\node at (-0.85,-0.35) [listsize2](x5){};
\node at (-0.5,-0.5) [listsize2](x6){};
\node at (-0.15,-0.35) [listsize3](x7){};
\node at (-1,0.5) [listsize3](x8){};
\node at (0.25,0.0) [listsize3](x9){};
\node at (-0.5,-0.7) [listsize3](x10){};
\coordinate (x11) at (-0,0.5);
\coordinate (x12) at (-0.5,0.7);
\coordinate (x13) at (-1.2,0);
\coordinate (x14) at (-1,-0.5);
\coordinate (x15) at (0.0,-0.5);
\coordinate (x16) at (-1.2,0.6);
\coordinate (x17) at (0.5,0);
\coordinate (x18) at (-0.5,-0.8);
\draw (x8) node[degree2]{} (x9) node[degree2]{} (x10) node[degree2]{};
\draw[realedges]  (x0)--(x1) (x0)--(x7) (x0)--(x9) (x1)--(x2) (x2)--(x3) (x3)--(x4) (x3)--(x8) (x4)--(x5) (x5)--(x6) (x6)--(x7) (x6)--(x10);
\draw[externaledges]  (x1)--(x11) (x2)--(x12) (x4)--(x13) (x5)--(x14) (x7)--(x15) (x8)--(x16) (x9)--(x17) (x10)--(x18);
\node at (-0.5,-0){8 };   
\end{tikzpicture}
&
\begin{tikzpicture}[scale=2.4]  
\node at (0.0,0.0) [listsize3](x0){};
\node at (0,0.3) [listsize4](x1){};
\node at (0.2,0.5) [listsize3](x2){};
\node at (0.5,0.5) [listsize2](x3){};
\node at (0.7,0.3) [listsize3](x4){};
\node at (0.7,0) [listsize2](x5){};
\node at (0.5,-0.2) [listsize2](x6){};
\node at (0.2,-0.2) [listsize2](x7){};
\node at (-0.2,-0.1) [listsize2](x8){};
\node at (-0.4,0) [listsize1](x9){};
\node at (-0.4,0.3) [listsize2](x10){};
\node at (-0.2,0.4) [listsize4](x11){};
\node at (0.5,0.75) [listsize3](x12){};
\node at (0.9,0) [listsize3](x13){};
\coordinate (x14) at (0.2,0.75);
\coordinate (x15) at (0.9,0.3);
\coordinate (x16) at (0.5,-0.4);
\coordinate (x17) at (0.2,-0.4);
\coordinate (x18) at (-0.2,-0.3);
\coordinate (x19) at (-0.6,0);
\coordinate (x20) at (-0.6,0.3);
\coordinate (x21) at (0.5,1);
\coordinate (x22) at (1.1,0);
\draw (x11) node[degree2]{} (x12) node[degree2]{} (x13) node[degree2]{};
\draw[realedges]  (x0)--(x1) (x0)--(x7) (x0)--(x8) (x1)--(x2) (x1)--(x11) (x2)--(x3) (x3)--(x4) (x3)--(x12) (x4)--(x5) (x5)--(x6) (x5)--(x13) (x6)--(x7) (x8)--(x9) (x9)--(x10) (x10)--(x11);
\draw[externaledges]  (x2)--(x14) (x4)--(x15) (x6)--(x16) (x7)--(x17) (x8)--(x18) (x9)--(x19) (x10)--(x20) (x12)--(x21) (x13)--(x22);
\node at (0.35,0.15){ 8 }; 
\node at (-.2,0.15){6};  
\end{tikzpicture}
&
\begin{tikzpicture}[scale=2.4]  
\node at (0.0,0.0) [listsize4](x0){};
\node at (0.0,0.3) [listsize4](x1){};
\node at (-0.2,0.5) [listsize3](x2){};
\node at (-0.5,0.5) [listsize3](x3){};
\node at (-0.7,0.3) [listsize2](x4){};
\node at (-0.7,0) [listsize2](x5){};
\node at (-0.5,-0.2) [listsize1](x6){};
\node at (-0.2,-0.2) [listsize2](x7){};
\node at (.2,-0.2) [listsize4](x8){};
\node at (0.4,0) [listsize2](x9){};
\node at (0.4,0.3) [listsize1](x10){};
\node at (0.2,0.5) [listsize2](x11){};
\node at (-.2,0.7) [listsize3](x12){};
\node at (-0.9,0.3) [listsize3](x13){};
\coordinate (x14) at (-0.5,0.7);
\coordinate (x15) at (-0.9,0);
\coordinate (x16) at (-0.5,-0.4);
\coordinate (x17) at (-0.2,-0.4);
\coordinate (x18) at (0.6,0);
\coordinate (x19) at (0.6,0.3);
\coordinate (x20) at (0.2,0.7);
\coordinate (x21) at (-0.2,0.95);
\coordinate (x22) at (-1.1,0.3);
\draw (x8) node[degree2]{} (x12) node[degree2]{} (x13) node[degree2]{};
\draw[realedges]  (x0)--(x1) (x0)--(x7) (x0)--(x8) (x1)--(x2) (x1)--(x11) (x2)--(x3) (x2)--(x12) (x3)--(x4) (x4)--(x5) (x4)--(x13) (x5)--(x6) (x6)--(x7) (x8)--(x9) (x9)--(x10) (x10)--(x11);
\draw[externaledges]  (x3)--(x14) (x5)--(x15) (x6)--(x16) (x7)--(x17) (x9)--(x18) (x10)--(x19) (x11)--(x20) (x12)--(x21) (x13)--(x22);
\node at (-0.35,0.15){ 8 }; 
\node at (0.2,.15){6};  
\end{tikzpicture}
\\
C_{8,6},  \hspace{3px} $\alon$ & C_{8,7},  \hspace{3px} $\greedy$ & C_{8,8},  \hspace{3px} $\alon$ & C_{8,9} \hspace{3px} $\alon$ \\
\end{array}
\]
\caption{Reducible configurations around an 8-face.}\label{fig-8x}
\end{figure}
\item Suppose that $\ell (f) = 8$; $\mu_0(f) = 2$,
and note the reducible configurations in Figure~\ref{fig-8x}.
By configuration {$C_{8,1}$}, $f$ cannot be incident to two 2-vertices.  

Suppose that $f$ is incident to one 2-vertex.  By configuration {$C_{8,2}$}, $f$ cannot have two nearby vertices.  If $f$ has one nearby 2-vertex, the distance between the incident 2-vertex and the nearby 2-vertex must be either 4 or 5. If the distance is 5, by {$C_{8,3}$}, the nearby 2-vertex is not needy.
If the distance is 4, then by {$C_{8,4}$} and {$C_{8,5}$} the nearby 2-vertex is not needy. Hence (R4) does not apply to $f$ and $f$ only loses at most charge 2 to the incident 2-vertex by (R1) or (R2). Therefore, the final charge of $f$ is nonnegative. 

Suppose that $f$ is not incident to any 2-vertex.  
By configuration {$C_{8,6}$}, if $f$ has four nearby 2-vertices, none of them can be needy.
If $f$ has three nearby 2-vertices, then they are configured in one of the two ways shown in Figure~\ref{8face3nearby}.  By configuration {$C_{8,7}$}, $f$ cannot have the configuration in Figure~\ref{8face3nearby1}.  If $f$ has the configuration in Figure~\ref{8face3nearby2}, then by {$C_{8,8}$} and {$C_{8,9}$}, neither $v_1$ nor $v_3$ is needy, thus $f$ loses charge at most 1 to $v_3$ by (R4).  If $f$ has fewer than three nearby 2-vertices, then it loses charge at most 2: 1 to each nearby 2-vertex by (R4).  Hence, if $\ell (f) = 8$, then $f$ has nonnegative final charge.

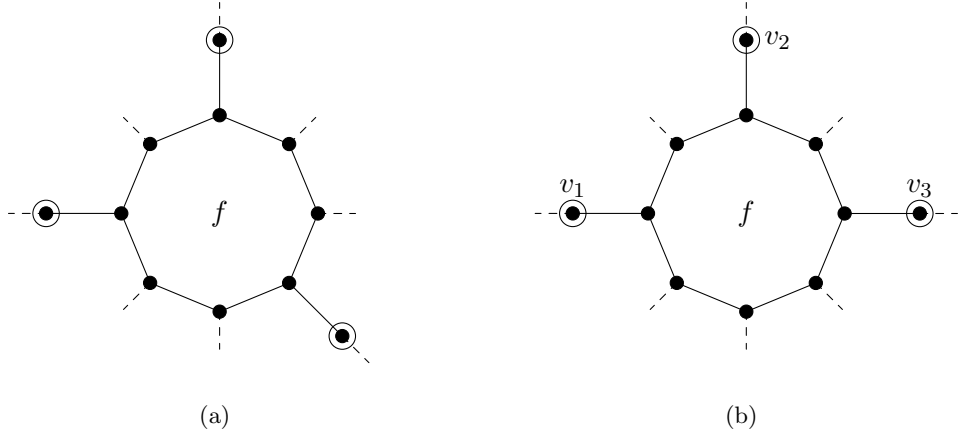
\begin{figure}
	\begin{center}
		\subfigure[\label{8face3nearby1}]{
		\begin{tikzpicture}[line cap=round,line join=round,>=triangle 45,x=1.0cm,y=1.0cm]
			\clip(0.693,0) rectangle (7.307,5.613);
			\draw (3.076,1.383)-- (4,1);
			\draw (4,1)-- (4.924,1.382);
			\draw (4.924,1.382)-- (5.307,2.306);
			\draw (5.307,2.306)-- (4.924,3.230);
			\draw (4.924,3.230)-- (4,3.613);
			\draw (4,3.613)-- (3.076,3.231);
			\draw (3.076,3.231)-- (2.693,2.307);
			\draw (2.693,2.307)-- (3.076,1.383);
			\draw (4.924,1.382)-- (5.631,0.675);
			\draw (4,3.613)-- (4,4.613);
			\draw (2.693,2.307)-- (1.693,2.307);
			\draw [dashed] (5.631,0.675)-- (5.985,0.321);
			\draw [dashed] (4,4.613)-- (4,5.113);
			\draw [dashed] (1.693,2.307)-- (1.193,2.307);
			\draw [dashed] (3.076,1.383)-- (2.722,1.029);
			\draw [dashed] (4,1)-- (4,0.5);
			\draw [dashed] (5.307,2.306)-- (5.807,2.306);
			\draw [dashed] (4.924,3.230)-- (5.278,3.584);
			\draw [dashed] (3.076,3.231)-- (2.722,3.584);
			\draw [fill=black] (4,1) circle (2.5pt);
			\draw [fill=black] (3.076,1.383) circle (2.5pt);
			\draw [fill=black] (4.924,1.382) circle (2.5pt);
			\draw [fill=black] (5.307,2.306) circle (2.5pt);
			\draw [fill=black] (4.924,3.230) circle (2.5pt);
			\draw [fill=black] (4,3.613) circle (2.5pt);
			\draw [fill=black] (3.076,3.231) circle (2.5pt);
			\draw [fill=black] (2.693,2.307) circle (2.5pt);
			\draw [fill=black] (1.693,2.307) circle (2.5pt);
			\draw (1.693,2.307) circle (5pt);
			\draw [fill=black] (4,4.613) circle (2.5pt);
			\draw (4,4.613) circle (5pt);
			\draw [fill=black] (5.631,0.675) circle (2.5pt);
			\draw (5.631,0.675) circle (5pt);
			\draw[color=black] (4,2.3) node {$f$};
		\end{tikzpicture} }
		\subfigure[\label{8face3nearby2}]{
		\begin{tikzpicture}[line cap=round,line join=round,>=triangle 45,x=1.0cm,y=1.0cm]
			\clip(0.693,0) rectangle (7.307,5.613);
			\draw (3.076,1.383)-- (4,1);
			\draw (4,1)-- (4.924,1.382);
			\draw (4.924,1.382)-- (5.307,2.306);
			\draw (5.307,2.306)-- (4.924,3.230);
			\draw (4.924,3.230)-- (4,3.613);
			\draw (4,3.613)-- (3.076,3.231);
			\draw (3.076,3.231)-- (2.693,2.307);
			\draw (2.693,2.307)-- (3.076,1.383);
			\draw (5.307,2.306)-- (6.307,2.306);
			\draw (4,3.613)-- (4,4.613);
			\draw (2.693,2.307)-- (1.693,2.307);
			\draw [dashed] (6.307,2.306)-- (6.807,2.306);
			\draw [dashed] (4,4.613)-- (4,5.113);
			\draw [dashed] (1.693,2.307)-- (1.193,2.307);
			\draw [dashed] (3.076,1.383)-- (2.722,1.029);
			\draw [dashed] (4,1)-- (4,0.5);
			\draw [dashed] (4.924,1.382)-- (5.278,1.029);
			\draw [dashed] (4.924,3.230)-- (5.278,3.584);
			\draw [dashed] (3.076,3.231)-- (2.722,3.584);
			\draw [fill=black] (4,1) circle (2.5pt);
			\draw [fill=black] (3.076,1.383) circle (2.5pt);
			\draw [fill=black] (4.924,1.382) circle (2.5pt);
			\draw [fill=black] (5.307,2.306) circle (2.5pt);
			\draw [fill=black] (4.924,3.230) circle (2.5pt);
			\draw [fill=black] (4,3.613) circle (2.5pt);
			\draw [fill=black] (3.076,3.231) circle (2.5pt);
			\draw [fill=black] (2.693,2.307) circle (2.5pt);
			\draw [fill=black] (1.693,2.307) circle (2.5pt);
			\draw (1.693,2.307) circle (5pt);
			\draw [fill=black] (4,4.613) circle (2.5pt);
			\draw (4,4.613) circle (5pt);
			\draw [fill=black] (6.307,2.306) circle (2.5pt);
			\draw (6.307,2.306) circle (5pt);
			\draw[color=black] (4,2.3) node {$f$};
			\draw[color=black] (1.693,2.64) node {$v_1$};
			\draw[color=black] (4.42,4.613) node {$v_2$};
			\draw[color=black] (6.307,2.64) node {$v_3$};
		\end{tikzpicture} }
		\caption{\label{8face3nearby} Two configurations of an 8-face $f$ with three nearby 2-vertices}
	\end{center}
\end{figure}

\begin{center}
	\begin{tabular}{|c|c|c|c|c|c|} \hline
\multicolumn{6}{|c|}{{\bf  $\ell(f) = 8$}} \\
\hline
		$|I(f)| / |N(f)|$ & 0 & 1 & 2 & 3 & 4 \\
		 \hline
		0 & EC & EC & EC & $C_{8,7},C_{8,8},C_{8,9}$&$C_{8,6}$ \\ \hline
		1 & EC & $C_{8,3},C_{8,4},C_{8,5}$  & $C_{8,2}$   & DR& DR  \\ \hline
		2 & $C_{8,1}$ &DR& DR& DR& DR \\ \hline
	\end{tabular}
\end{center}

\begin{figure}
\[
\begin{array}{cccc} 	
\begin{tikzpicture}[scale=2.5]  
\node at (0.0,.15) [listsize3](x0){};
\node at (-.2,0.35) [listsize2](x1){};
\node at (-0.4,0.35) [listsize2](x2){};
\node at (-0.6,0.2) [listsize3](x3){};
\node at (-0.6,-0.0) [listsize5](x4){};
\node at (-0.6,-0.2) [listsize3](x5){};
\node at (-0.4,-0.35) [listsize2](x6){};
\node at (-0.2,-0.35) [listsize1](x7){};
\node at (0,-0.15) [listsize2](x8){};
\node at (-0.8,0.2) [listsize2](x9){};
\node at (-.9,0) [listsize1](x10){};
\node at (-0.8,-0.2) [listsize2](x11){};
\coordinate (x12) at (-.2,0.55);
\coordinate (x13) at (-0.4,0.55);
\coordinate (x14) at (-0.4,-0.55);
\coordinate (x15) at (-0.2,-0.55);
\coordinate (x16) at (0.2,-0.15);
\coordinate (x17) at (-0.8,0.4);
\coordinate (x18) at (-1.1,0);
\coordinate (x19) at (-0.8,-0.4);
\draw (x0) node[degree2]{} (x4) node[degree2]{};
\draw[realedges]  (x0)--(x1) (x0)--(x8) (x1)--(x2) (x2)--(x3) (x3)--(x4) (x3)--(x9) (x4)--(x5) (x5)--(x6) (x5)--(x11) (x6)--(x7) (x7)--(x8) (x9)--(x10) (x10)--(x11);
\draw[externaledges]  (x1)--(x12) (x2)--(x13) (x6)--(x14) (x7)--(x15) (x8)--(x16) (x9)--(x17) (x10)--(x18) (x11)--(x19);
\node at (-0.3,-0.0){9  };  
\node at (-.75,0){6};  
\end{tikzpicture}
&
\begin{tikzpicture}[scale=2]  
\node at (0.0,0.0) [listsize3](x0){};
\node at (-0.1,0.3) [listsize2](x1){};
\node at (-0.4,0.4) [listsize2](x2){};
\node at (-0.7,0.3) [listsize2](x3){};
\node at (-0.8,0) [listsize3](x4){};
\node at (-0.75,-0.3) [listsize2](x5){};
\node at (-0.52,-0.45) [listsize2](x6){};
\node at (-0.27,-0.45) [listsize1](x7){};
\node at (-0.05,-0.3) [listsize2](x8){};
\node at (-0.85,0.5) [listsize3](x9){};
\node at (-0.9,-0.5) [listsize3](x10){};
\coordinate (x11) at (0.1,0.43);
\coordinate (x12) at (-0.4,0.65);
\coordinate (x13) at (-1,0);
\coordinate (x14) at (-0.52,-0.7);
\coordinate (x15) at (-0.27,-0.7);
\coordinate (x16) at (.1,-0.43);
\coordinate (x17) at (-1,.65);
\coordinate (x18) at (-1.05,-0.67);
\draw (x0) node[degree2]{} (x9) node[degree2]{} (x10) node[degree2]{};
\draw[realedges]  (x0)--(x1) (x0)--(x8) (x1)--(x2) (x2)--(x3) (x3)--(x4) (x3)--(x9) (x4)--(x5) (x5)--(x6) (x5)--(x10) (x6)--(x7) (x7)--(x8);
\draw[externaledges]  (x1)--(x11) (x2)--(x12) (x4)--(x13) (x6)--(x14) (x7)--(x15) (x8)--(x16) (x9)--(x17) (x10)--(x18);
\node at (-0.4,0.0){ 9};   
\end{tikzpicture}
&
\begin{tikzpicture}[scale=2]  
\node at (0.0,0.0) [listsize3](x0){};
\node at (-0.1,0.3) [listsize2](x1){};
\node at (-0.4,0.4) [listsize2](x2){};
\node at (-0.7,0.3) [listsize2](x3){};
\node at (-0.8,0) [listsize2](x4){};
\node at (-0.75,-0.3) [listsize2](x5){};
\node at (-0.52,-0.45) [listsize2](x6){};
\node at (-0.27,-0.45) [listsize2](x7){};
\node at (-0.05,-0.3) [listsize2](x8){};
\node at (-.85,.5) [listsize3](x9){};
\node at (-0.52,-0.7) [listsize3](x10){};
\coordinate (x11) at (0.1,0.43);
\coordinate (x12) at (-0.4,0.65);
\coordinate (x13) at (-1,0);
\coordinate (x14) at (-0.9,-0.5);
\coordinate (x15) at (-.27,-0.7);
\coordinate (x16) at (.1,-0.43);
\coordinate (x17) at (-1,0.65);
\coordinate (x18) at (-.52,-.9);
\draw (x0) node[degree2]{} (x9) node[degree2]{} (x10) node[degree2]{};
\draw[realedges]  (x0)--(x1) (x0)--(x8) (x1)--(x2) (x2)--(x3) (x3)--(x4) (x3)--(x9) (x4)--(x5) (x5)--(x6) (x6)--(x7) (x6)--(x10) (x7)--(x8);
\draw[externaledges]  (x1)--(x11) (x2)--(x12) (x4)--(x13) (x5)--(x14) (x7)--(x15) (x8)--(x16) (x9)--(x17) (x10)--(x18);
\node at (-0.4,0){ 9 };   
\end{tikzpicture}
&
\begin{tikzpicture}[scale=2]  
\node at (0.0,0.0) [listsize2](x0){};
\node at (-0.15,-0.24) [listsize2](x1){};
\node at (-0.15,-0.5) [listsize2](x2){};
\node at (0.1,-0.7) [listsize2](x3){};
\node at (0.4,-0.7) [listsize3](x4){};
\node at (0.65,-0.5) [listsize2](x5){};
\node at (0.65,-0.24) [listsize3](x6){};
\node at (0.5,0.0) [listsize2](x7){};
\node at (0.25,0.1) [listsize3](x8){};
\node at (0.0,-.9) [listsize3](x9){};
\node at (0.83,-0.63) [listsize3](x10){};
\node at (0.65,0.2) [listsize3](x11){};
\node at (-0.2,0.15) [listsize3](x12){};
\coordinate (x13) at (-0.4,-0.24);
\coordinate (x14) at (-0.3,-0.65);
\coordinate (x15) at (0.5,-0.9);
\coordinate (x16) at (.9,-0.24);
\coordinate (x17) at (0.25,0.35);
\coordinate (x18) at (-.10,-1.09);
\coordinate (x19) at (1.03,-0.74);
\coordinate (x20) at (0.77,0.4);
\coordinate (x21) at (-0.39,0.29);
\draw (x9) node[degree2]{} (x10) node[degree2]{} (x11) node[degree2]{} (x12) node[degree2]{};
\draw[realedges]  (x0)--(x1) (x0)--(x8) (x0)--(x12) (x1)--(x2) (x2)--(x3) (x3)--(x4) (x3)--(x9) (x4)--(x5) (x5)--(x6) (x5)--(x10) (x6)--(x7) (x7)--(x8) (x7)--(x11);
\draw[externaledges]  (x1)--(x13) (x2)--(x14) (x4)--(x15) (x6)--(x16) (x8)--(x17) (x9)--(x18) (x10)--(x19) (x11)--(x20) (x12)--(x21);
\node at (0.25,-0.3){ 9 };   
\end{tikzpicture}
\\
C_{9,1}, \hspace{3px} $\greedy$ & C_{9,2}, \hspace{3px} $\greedy$ & C_{9,3}, \hspace{3px} $\greedy$ & C_{9,4},  \hspace{3px} $\greedy$\\
\end{array}
\]
\caption{Reducible configurations around a 9-face.}\label{fig-9x}
\end{figure}
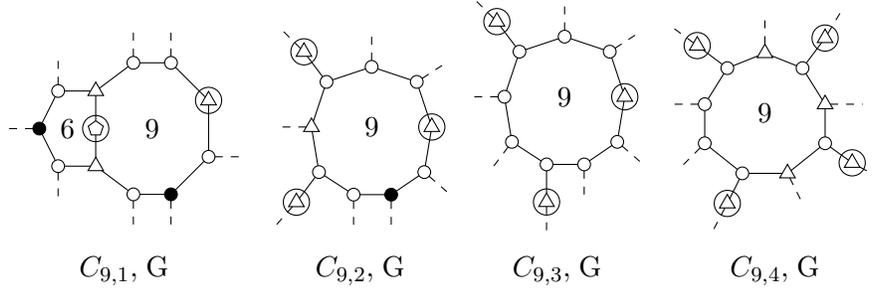
\item  Suppose that $\ell (f) = 9$; $\mu_0(f) = 3$,
and note the reducible configurations in Figure~\ref{fig-9x}.

  If $f$ is incident to two 2-vertices $v_1$ and $v_2$, then by configuration {$C_{9,1}$} neither $v_1$ nor $v_2$ can be incident to a 6-face, thus $f$ only loses charge 1 to each 2-vertex by (R2).  
  By Claim~\ref{claim1}, $f$ does not have any nearby vertices. Hence the final charge of $f$ is nonnegative.
  
  Suppose that $f$ is incident to one 2-vertex.  By configurations {$C_{9,2}$} and {$C_{9,3}$}, $f$ cannot have two nearby 2-vertices.  Since $f$ has at most one 2-vertex $v$, it loses charge at most 3; at most 2 to the incident 2-vertex by (R1) and at most 1 to $v$ by (R4).  Hence, if $f$ is incident to one 2-vertex, its final charge is nonnegative.  
  
 Suppose that $f$ is not incident to any 2-vertex.  By configuration {$C_{9,4}$}, $f$ cannot have four nearby 2-vertices.  If $f$ has at most three nearby 2-vertices, then it loses charge at most 3: at most 1 to each nearby 2-vertex by (R4).  Hence, if $\ell (f) = 9$, then $f$ has nonnegative final charge.
  \begin{center}
	\begin{tabular}{|c|c|c|c|c|c|} \hline
\multicolumn{6}{|c|}{{\bf  $\ell(f) = 9$}} \\
\hline
		$|I(f)| / |N(f)|$ & 0 & 1 & 2 & 3 & 4 \\
		 \hline
		0 & EC & EC & EC & EC & $C_{9,4}$ \\ \hline
		1 & EC & EC & $C_{9,2},C_{9,3}$   & DR& DR  \\ \hline
		2 & $C_{9,1}$ &DR & DR& DR&  DR \\ \hline
	\end{tabular}
\end{center}

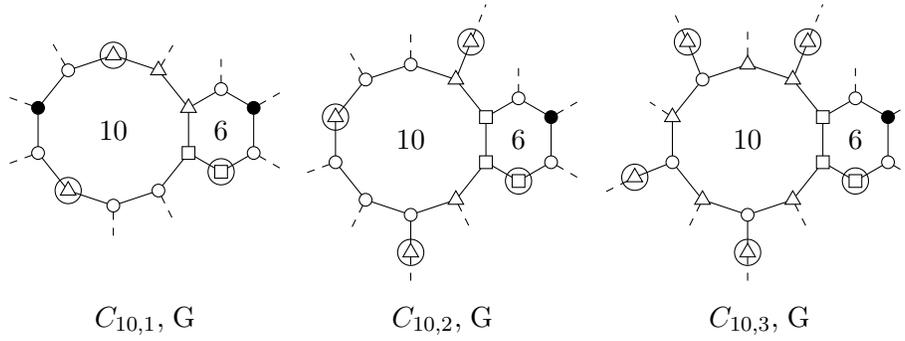
\begin{figure}
\[
\begin{array}{ccc}	
\begin{tikzpicture}[scale=2]  
\node at (0.0,0.0) [listsize3](x0){};
\node at (0.3,-0.1) [listsize3](x1){};
\node at (0.5,-0.35) [listsize3](x2){};
\node at (0.5,-0.65) [listsize4](x3){};
\node at (0.3,-0.9) [listsize2](x4){};
\node at (0,-1) [listsize2](x5){};
\node at (-0.3,-0.9) [listsize3](x6){};
\node at (-0.5,-0.65) [listsize2](x7){};
\node at (-0.5,-0.35) [listsize1](x8){};
\node at (-0.3,-0.1) [listsize2](x9){};
\node at ($(x3)+(-30:0.25)$) [listsize4](x10){};
\draw(x2)  ++(30:0.25) node[listsize2](y1){}  
 ++(-30:0.25) node[listsize1](y2){}  
 ++(270:0.3) node[listsize2](y3){}
 ;
\draw[externaledges]
(y1) -- ++(90:0.2)
(y2) -- ++(30:0.2)
(y3) -- ++(-30:0.2)
(x9) -- ++(120:0.2)
(x7) -- ++(200:0.2)
(x1) -- ++(60:0.2)
(x5) -- ++(270:0.2)
(x8) -- ++(160:0.2)
;
\coordinate (x15) at (0,0.25);
\coordinate (x16) at (0.65,-0.25);
\coordinate (x17) at (.4,-1.1);
\coordinate (x18) at (-.4,-1.1);
\coordinate (x21) at (0.5,0.4);
\coordinate (x22) at (-0.5,.37);
\coordinate (x23) at (-.95,-.85);
\coordinate (x24) at (0,-1.45);
\draw (x10) node[degree2]{} (x6) node[degree2]{} (x0) node[degree2]{};
\draw[realedges]  (x0)--(x1) (x0)--(x9) (x1)--(x2)  (x2)--(x3) (x3)--(x4) (x3)--(x10) (x4)--(x5) (x5)--(x6) (x6)--(x7) (x7)--(x8)  (x8)--(x9)  (x2)--(y1) (y1)--(y2) (y2)--(y3) (y3)--(x10);
\draw[externaledges]   (x4)--(x17) ;
\node at (0,-0.5){ 10 };   
\node at ($(y2)!0.5!(x3)$){6};   
\node at  (x24){};
\end{tikzpicture}
& 	
\begin{tikzpicture}[scale=2]  
\node at (0.0,0.0) [listsize2](x0){};
\node at (0.3,-0.1) [listsize3](x1){};
\node at (0.5,-0.35) [listsize4](x2){};
\node at (0.5,-0.65) [listsize4](x3){};
\node at (0.3,-0.9) [listsize3](x4){};
\node at (0,-1) [listsize2](x5){};
\node at (-0.3,-0.9) [listsize2](x6){};
\node at (-0.5,-0.65) [listsize2](x7){};
\node at (-0.5,-0.35) [listsize3](x8){};
\node at (-0.3,-0.1) [listsize2](x9){};
\node at ($(x3)+(-30:0.25)$) [listsize4](x10){};
\node at (0.4,0.15) [listsize3](x11){};
\node at (0,-1.25)  [listsize3](x14){};
\draw(x2)  ++(30:0.25) node[listsize2](y1){}  
 ++(-30:0.25) node[listsize1](y2){}  
 ++(270:0.3) node[listsize2](y3){}
 ;
\draw[externaledges]
(y1) -- ++(90:0.2)
(y2) -- ++(30:0.2)
(y3) -- ++(-30:0.2)
(x9) -- ++(120:0.2)
(x7) -- ++(200:0.2)
;
\coordinate (x15) at (0,0.25);
\coordinate (x16) at (0.65,-0.25);
\coordinate (x17) at (.4,-1.1);
\coordinate (x18) at (-.4,-1.1);
\coordinate (x21) at (0.5,0.4);
\coordinate (x22) at (-0.5,.37);
\coordinate (x23) at (-.95,-.85);
\coordinate (x24) at (0,-1.45);
\draw (x10) node[degree2]{} (x11) node[degree2]{} (x8) node[degree2]{} (x14) node[degree2]{};
\draw[realedges]  (x0)--(x1) (x0)--(x9) (x1)--(x2) (x1)--(x11) (x2)--(x3) (x3)--(x4) (x3)--(x10) (x4)--(x5) (x5)--(x6) (x5)--(x14) (x6)--(x7) (x7)--(x8)  (x8)--(x9)  (x2)--(y1) (y1)--(y2) (y2)--(y3) (y3)--(x10);
\draw[externaledges]  (x0)--(x15) (x4)--(x17) (x6)--(x18)  (x11)--(x21)  (x14)--(x24);
\node at (0,-0.5){ 10 };   
\node at ($(y2)!0.5!(x3)$){6};   
\end{tikzpicture}
&
\begin{tikzpicture}[scale=2]  
\node at (0.0,0.0) [listsize3](x0){};
\node at (0.3,-0.1) [listsize3](x1){};
\node at (0.5,-0.35) [listsize4](x2){};
\node at (0.5,-0.65) [listsize4](x3){};
\node at (0.3,-0.9) [listsize3](x4){};
\node at (0,-1) [listsize2](x5){};
\node at (-0.3,-0.9) [listsize3](x6){};
\node at (-0.5,-0.65) [listsize2](x7){};
\node at (-0.5,-0.35) [listsize3](x8){};
\node at (-0.3,-0.1) [listsize2](x9){};
\node at ($(x3)+(-30:0.25)$) [listsize4](x10){};
\node at (0.4,0.15) [listsize3](x11){};
\node at (-0.4,.15) [listsize3](x12){};
\node at (-0.75,-0.75) [listsize3](x13){};
\node at (0,-1.25)  [listsize3](x14){};
\draw(x2)  ++(30:0.25) node[listsize2](y1){}  
 ++(-30:0.25) node[listsize1](y2){}  
 ++(270:0.3) node[listsize2](y3){}
 ;
\draw[externaledges]
(y1) -- ++(90:0.2)
(y2) -- ++(30:0.2)
(y3) -- ++(-30:0.2)
;
\coordinate (x15) at (0,0.25);
\coordinate (x16) at (0.65,-0.25);
\coordinate (x17) at (.4,-1.1);
\coordinate (x18) at (-.4,-1.1);
\coordinate (x19) at (-0.65,-0.25);
\coordinate (x21) at (0.5,0.4);
\coordinate (x22) at (-0.5,.37);
\coordinate (x23) at (-.95,-.85);
\coordinate (x24) at (0,-1.45);
\draw (x10) node[degree2]{} (x11) node[degree2]{} (x12) node[degree2]{} (x13) node[degree2]{} (x14) node[degree2]{};
\draw[realedges]  (x0)--(x1) (x0)--(x9) (x1)--(x2) (x1)--(x11) (x2)--(x3) (x3)--(x4) (x3)--(x10) (x4)--(x5) (x5)--(x6) (x5)--(x14) (x6)--(x7) (x7)--(x8) (x7)--(x13) (x8)--(x9) (x9)--(x12) (x2)--(y1) (y1)--(y2) (y2)--(y3) (y3)--(x10);
\draw[externaledges]  (x0)--(x15) (x4)--(x17) (x6)--(x18) (x8)--(x19) (x11)--(x21) (x12)--(x22) (x13)--(x23) (x14)--(x24);
\node at (0,-0.5){ 10 };   
\node at ($(y2)!0.5!(x3)$){6};   
\end{tikzpicture}
\\
 C_{10,1},  \hspace{3px} $\greedy$   & C_{10,2},  \hspace{3px} $\greedy$  & C_{10,3},  \hspace{3px} $\greedy$  \\
\end{array}
\]
\caption{Reducible configurations around a 10-face.}\label{fig-10x}
\end{figure}

\item  Suppose that $\ell (f) = 10$; $\mu_0(f) = 4$, and note the reducible configurations in Figure~\ref{fig-10x}.

  If $f$ is incident to two 2-vertices, then by configuration {$C_{10,1}$}, $f$ has no nearby needy 2-vertices.
  Hence $f$ twice loses charge at most 2 by (R1) or (R2).  
  Hence the final charge of $f$ is nonnegative.
  
  Suppose that $f$ is incident to one 2-vertex.  By configuration {$C_{10,2}$}, $f$ can have at most two nearby needy 2-vertices.
  Hence $f$  loses at most charge 2 by (R1) or (R2) once and loses at most twice charge 1 by (R4). 
  Hence the final charge of $f$ is nonnegative.
  
   Suppose that $f$ is not incident to any 2-vertex.  
   By configuration {$C_{10,3}$}, if $f$ has five nearby 2-vertices, none of them can be needy.
   Hence $f$ loses at most four times charge 1 by (R4). 
   Hence, if $\ell (f) = 10$, then $f$ has nonnegative final charge.
  \begin{center}
	\begin{tabular}{|c|c|c|c|c|c|c|} \hline
\multicolumn{7}{|c|}{{\bf $\ell(f) = 10$}} \\
\hline
		$|I(f)|/ |N(f)|$ & 0 & 1 & 2 & 3 & 4&5 \\
		 \hline
		0 & EC & EC & EC & EC & EC& $C_{10,3}$ \\ \hline
		1 & EC & EC  & EC  & $C_{10,2}$ & DR& DR \\ \hline
		2 & EC &$C_{10,1}$ & DR&DR&  DR &DR \\ \hline
	\end{tabular}
\end{center}

\item Suppose that $\ell(f) \geq 11$. We define sets of edges that are close to 2-vertices.
For every 2-vertex $u$, define
\[
W_u = \{e \in E(G) :  \exists v \in N_G(u), v \in e   \}.
\]
As seen in Figure~\ref{fig3}, if $u \in I(f)$ then $\left|W_u \cap f\right| \geq 4$ since $u$ has at least two distinct neighbors incident with $f$ and these neighbors are incident to four distinct edges of $f$. Similarly if $v \in N(f)$, then $\left|W_v \cap f\right| \geq 2$.
By Claim~\ref{claim1}, $W_{u} \cap W_{v} = \emptyset$ for any two distinct 2-vertices $u$ and $v$.
Hence $\ell(f) \geq 4 |I(f)| + 2|N(f)|$.
Every vertex in $I(f)$ receives charge at most 2 by (R1) and each vertex in $N(f)$ receives charge at most 1 by (R4). 

If $f$ is not the outer face, we show that the final charge of $f$ is nonnegative as follows:
\begin{align*}
\mu_5(f) &\geq \mu_0(f) - 2|I(f)| - |N(f)|  = \ell(f) - 6 - 2|I(f)| - |N(f)| \\
              &\geq  \left(\left\lceil\frac{\ell(f)}{2}\right\rceil - 6\right)  + \left(\left\lfloor \frac{\ell(f)}{2} \right\rfloor - 2|I(f)| - |N(f)|\right) \geq 0.
\end{align*}
If $f$ is the outer face $F_o$, we need a slightly better estimate. 
Notice that $F_o-\mathcal{P}$ is a cycle of length $\ell(F_o)-2|\mathcal{P}|$. 
Let $B$ be the set of bad vertices in $G$. 
Recall (R5) applies only to bad vertices.
Notice  $\ell(F_o)-2|\mathcal{P}| \geq 4 |I(f)| + 2|N(f)|$.
Rules (R1), (R2), (R4), and (R5) may apply and the computation of the final charge is
\begin{align*}
\mu_5(F_o) &\geq \mu_0(F_o) - 2|I(F_o)| - |N(F_o)| - |B| = \ell(F_o) - 5 -|\mathcal{P}| - 2|I(F_o)| - |N(F_o)| - |B| \\
              &\geq  \left(\left\lceil\frac{\ell(F_o)}{2}\right\rceil - 5 - |B|\right)  + \left(\left\lfloor \frac{\ell(F_o)}{2} \right\rfloor -|\mathcal{P}| - 2|I(F_o)| - |N(F_o)|\right) \geq 0.
\end{align*}

Observe that $C_{2,2}^\star$  can appear only once in $G$, so $|B| \leq 1$.
Hence, if $\ell (f) \geq 11$, then $f$ has a nonnegative final charge.

\begin{figure}[h!]
\centering
\begin{tikzpicture}[scale=2]
\draw (0,0) ellipse (2cm and .5cm);
\node at (-1.78201304837674, 0.226995249869773)[listsize1](x0){};
\node at (-0.907980999479093, 0.445503262094184)[listsize1](x1){};
\node at (-1.41421356237310, 0.353553390593274)[listsize1](x2){};
\node at (-1.49, 0.627106781186548) [listsize2](x3){$v$};

\node at (1.41421356237310, 0.353553390593274) [listsize2,fill=white](x4){$u$};
\node at (0.907980999479094, 0.445503262094184) [listsize1](x5){};
\node at (0.312868930080462, 0.493844170297569) [listsize1](x6){};
\node at (1.78201304837674, 0.226995249869773) [listsize1](x7){};
\node at (1.97537668119028, 0.0782172325201154) [listsize1](x8){};
\node at (1.3, 0.353553390593274) (x9){};
\node at (-1.35, 0.45) (x10){};
\draw (x3) node[degree2]{} (x4) node[degree2]{};
\draw[realedges]  (x2)--(x3);	 
\node at (0,0) {$f$}; 
\draw[dashed,rotate=10] (x10) ellipse (.65cm and .3cm);
\draw[dashed,rotate=-10] (x9) ellipse (1.1cm and .25cm);
\draw[-,line width=1.8pt] (x0)--(x2) (x2)--(x1) (x4)--(x5) (x5)--(x6) (x4)--(x7) (x7)--(x8);

\end{tikzpicture}
\caption{Edges on face $f$ in $W_u$ and $W_v$.}
\label{fig3}
\end{figure}
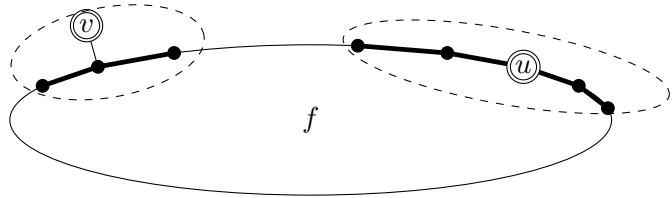

\end{itemize} 

We conclude that  $\sum_{v \in V(G)}\mu_0(v) + \sum_{f \in F(G)} \mu_0(f)  \geq 0$, which is a contradiction with \eqref{eqneg}.
This concludes the proof of Theorem~\ref{thm:ICN:realresult}.


\section{Conclusion and Future Work}\label{sec:future}

In this paper, we have shown through the method of discharging that subcubic planar graphs with girth at least 6 are injectively 5-choosable. This result improves several known bounds on the injective chromatic number and injective choosability in particular cases.
However, it leaves the most interesting conjecture about injective 5-coloring of planar graphs open.

\begin{conjecture}[Chen, Hahn, Raspaud and Wang~\cite{ChenRapsaud}]\label{conj:strong}
If a planar graph $G$ has $\Delta(G)=3$, then $\chi_i(G)\leq 5$.
\end{conjecture}

We believe it might be possible to answer the following question in the affirmative.
\begin{question}
Is there a planar graph $G$ with $\Delta(G)=4$, $g(G)\geq 6$, and $\chi_i^{\ell}(G)=6$?
\end{question}

We are not aware of counterexamples to the following problems, which are closely related  to our result.
\begin{problem}
If a planar graph $G$ has $\Delta(G)=4$ and $g(G)\geq 6$, then $\chi_i^{\ell}(G)\leq 6$.
\end{problem}
\begin{problem}
If a planar graph $G$ has $\Delta(G)=3$ and $g(G)\geq 5$, then $\chi_i^{\ell}(G)\leq 5$.
\end{problem}

\noindent 
Note that these conjectures on injective choosability are analogous to the conjecture on injective colorability of Chen, Hahn, Raspaud and Wang~\cite{ChenRapsaud}, without requiring a girth restriction.

\section*{Acknowledgements}
We gratefully acknowledge financial support for this research from the following grants and organizations: NSF-DMS Grants 1604458, 1604773, 1604697 and 1603823 (all authors), NSF 1450681 (B. Brimkov), NSF DMS-1600390 (B. Lidick\'{y}).\\
We would like to thank Derrick Stolee for letting us reuse some of his graph coloring code, Stephen Hartke for spotting a mistake in our reducibility code in the earlier version of the paper and to the anonymous referee for valuable comments.

\bibliographystyle{abbrv}
\bibliography{Coloring}

\end{document}